\documentclass[12pt]{article}
\usepackage{amssymb,amsmath,amsthm,tikz,multirow, mathdots}
\usepackage{enumerate} 
\usepackage{hyperref} 
\usepackage [autostyle, english = british]{csquotes}
\usepackage [english]{babel}
\usepackage{mathtools}
\usepackage{bm}
\usepackage{authblk} 

\usepackage{comment} 

\usetikzlibrary{calc, arrows, arrows.meta, math} 

\newtheorem*{theorem*}{Theorem}

\theoremstyle{definition}
\newtheorem*{definition*}{Definition}
\newtheorem*{case*}{Case}
\newtheorem*{subcase*}{Subcase}
\newtheorem*{subsubcase*}{Subsubcase}

\theoremstyle{plain}
\newtheorem{thm}{Theorem}[section]
\newtheorem{lem}[thm]{Lemma}
\newtheorem{prop}[thm]{Proposition}

\theoremstyle{definition}

\theoremstyle{remark}

\numberwithin{equation}{section}

\newcommand{\AVC}{\text{AVC}} 

\newcommand{\bvert}{\vrule width 2pt}

\newcommand{\quotes}[1]{``#1''} 

\newcommand{\arcThroughThreePoints}[4][]{
\coordinate (middle1) at ($(#2)!.5!(#3)$);
\coordinate (middle2) at ($(#3)!.5!(#4)$);
\coordinate (aux1) at ($(middle1)!1!90:(#3)$);
\coordinate (aux2) at ($(middle2)!1!90:(#4)$);
\coordinate (center) at ($(intersection of middle1--aux1 and middle2--aux2)$);
\draw[#1] 
 let \p1=($(#2)-(center)$),
      \p2=($(#4)-(center)$),
      \n0={veclen(\p1)},       
      \n1={atan2(\y1,\x1)}, 
      \n2={atan2(\y2,\x2)},
      \n3={\n2>\n1?\n2:\n2+360}
    in (#2) arc(\n1:\n3:\n0);
}

\providecommand{\keywords}[1]{\noindent \textit{Keywords:} #1}
\providecommand{\subject}[1]{\noindent \textit{Mathematics Subject Classification:} #1}

\title{Dihedral Tilings of the Sphere by Regular Polygons and Quadrilaterals I: Quadrilaterals with Equal Opposite Edges}
\author[]{Hoi Ping LUK}
\affil[]{The Hong Kong University of Science \& Technology} 
\affil[]{email: hoi@connect.ust.hk}

\begin{document}
\maketitle

\begin{abstract} We classify the dihedral edge-to-edge tilings of the sphere by regular polygons and quadrilaterals with equal opposite edges (edge configuration $xyxy$). \\

 \\

\subject{05B45, 52C20, 51M09, 51M20}
\end{abstract}

\section{Introduction}

The history of the studies on spherical tilings can be traced back to Plato ($5$ Platonic solids) and Archimedes ($13$ Archimedean solids). Recently, there are two major breakthroughs in the research on spherical tilings. One of them is the classification of tilings of the sphere by regular polygons \cite{aehj, joh, zal}. Another one is the classification of monohedral edge-to-edge tilings of the sphere, which was pioneered by Sommerville \cite{som} and completed through a collective effort \cite{awy, ay, cl, cl2, cly, gsy, ua, wy, wy2}. 

This series is the start of a classification of the dihedral edge-to-edge tilings of the sphere. In a monohedral tiling, every tile is congruent to one polygon, which we call a {\em prototile}. In a dihedral tiling, there are two prototiles. One of the prototiles in this series is a regular polygon. 

This paper is the first of the series. The two prototiles are one regular polygon ($m$-gon with edge configuration $x^m$, where $m\ge3$) and one quadrilateral with equal opposite edges $x,y$ (edge configuration $xyxy$, where $x \neq y$). For each $m\ge3$, the regular polygon has angles $\alpha$ and the quadrilateral has equal opposite angles $\beta, \gamma$. The prototiles are depicted in Figure \ref{Fig-a4-abab-angles}, where the quadrilateral is unshaded and the regular polygons are shaded. Throughout our discussion, the shaded tiles are always regular polygons. We assume that the degree of each vertex is $\ge3$. 

\begin{figure}[h!] 
\centering
\begin{tikzpicture}

\tikzmath{
\s=1;
\r=0.8;
\gon=4;
\th=360/\gon;
\x=\r*cos(\th/2);
\R=\r;
\ph=360/5;
\xx=\r*cos(\ph/2);
}

\begin{scope}[] 

\foreach \a in {0,...,3} {
\draw[rotate=\th*\a]
	(90-0.5*\th:\r) -- (90+0.5*\th:\r) 
;

}

\draw[line width=2]
	(90+0.5*\th:\r)  -- (90+1.5*\th:\r)
	(90-0.5*\th:\r)  -- (90-1.5*\th:\r)
;

\foreach \c in {0,1} {

\node at (0.5*\th+2*\th*\c:0.65*\r) {\small $\gamma$};

\node at (\th+2*\th*\c: 1*\r) {\small $x$}; 
\node at (0+2*\th*\c: 1*\r) {\small $y$}; 

}

\node at (1.55*\th:0.625*\r) {\small $\beta$};
\node at (3.5*\th:0.625*\r) {\small $\beta$};

\end{scope}

\begin{scope}[xshift=2.5*\s cm, yshift=-0.15*\s cm] 

\fill[gray!40]
	(90:\r) -- (90+120:\r) --  (90+2*120:\r) -- cycle
;

\foreach \a in {0,1,2} {

\draw[rotate=\a*120]
	(90:\r) -- (90+120:\r)
;

\node at (90+\a*120:0.55*\r) {\small $\alpha$};

\node at (90-0.5*120+\a*120:0.8*\r) {\small $x$};

}

\end{scope}

\begin{scope}[xshift=5*\s cm] 

\fill[gray!40]
	(0.5*\th:\r) -- (1.5*\th:\r) -- (2.5*\th:\r)  -- (3.5*\th:\r) -- cycle
;

\foreach \a in {0,...,3} {
\draw[rotate=\th*\a]
	(90-0.5*\th:\r) -- (90+0.5*\th:\r) 
;

}

\foreach \c in {0,...,3} {

\node at (0.5*\th+\th*\c:0.65*\r) {\small $\alpha$};

\node at (\th+\th*\c: 1*\r) {\small $x$}; 

}

\end{scope}

\begin{scope}[xshift=7.5*\s cm] 

\fill[gray!40]
	(90:\R) -- (90+\ph:\R) -- (90+2*\ph:\R) -- (90+3*\ph:\R) -- (90+4*\ph:\R) -- cycle
;

\foreach \a in {0,...,4} {

\draw[rotate=\a*\ph]
	(90-1*\ph:\R) -- (90:\R)
;

\node at (90+\a*\ph:0.7*\R) {\small $\alpha$};

\node at (\th-0.5*\ph+\ph*\a: 1.1*\r) {\small $x$}; 

}

\node at (2.5*\r, 0) {\Large $\cdots$};

\end{scope}

\end{tikzpicture}
\caption{The quadrilateral with edges $x,y$ and angles $\beta,\gamma$; and the regular polygons with edges $x$ and angles $\alpha$}
\label{Fig-a4-abab-angles}
\end{figure}
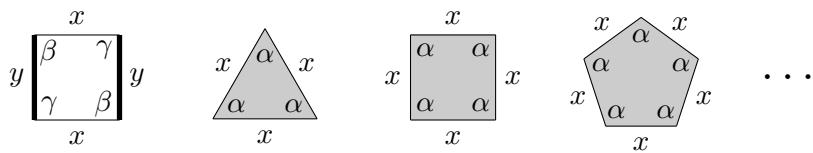

In the subsequent works \cite{cl3, luk}, edges of both prototiles have equal length. Other studies on dihedral tilings of the sphere with an extra assumption (folding type) can be seen in \cite{as, br, bs, brs}. 

For simplicity, by {\em the quadrilateral} we mean the forementioned quadrilateral and by {\em dihedral tilings} we mean those by regular polygons and the quadrilaterals with equal opposite edges. We may omit mentioning dihedral tilings when the context is clear and obvious.

The main result is given below, where $f$ denotes the number of tiles.

\begin{theorem*} The dihedral tilings of the sphere by regular polygons with gonality $m\ge3$ and quadrilaterals with equal opposite edges are,
\begin{enumerate}[I.]
\item Prism type: one infinite family of tilings with $f=m+2$, and vertex $\{ \alpha\beta\gamma \}$;
\item Sporadic type: specific triangle subdivisions of deformed prism type tilings with $3 \le m \le 6$.
\end{enumerate}
\end{theorem*}

The prism type tilings are illustrated in Figures \ref{Fig-a3/am-abab-Tilings}. The sporadic tilings are illustrated in Figure \ref{Fig-a3-abab-Sporadic-Tilings}.


\begin{figure}[h!]
\centering
\begin{tikzpicture}[scale=1]

\tikzmath{
\s=1;
\ss=0.8;
\sss=0.75;
}

\begin{scope}[]

\tikzmath{
\r=0.35;
\th=360/3;
\e=1.325;
}

\fill[gray!40]
	(90:\e*3*\r) -- (90+\th:\e*3*\r) -- (90+2*\th:\e*3*\r) -- cycle
;

\fill[white]
	(90:3*\r) -- (90+\th:3*\r) -- (90+2*\th:3*\r) -- cycle
;

\fill[gray!40]
	(90:\r) -- (90+\th:\r) -- (90+2*\th:\r) -- cycle
;

\foreach \t in {0,...,2} {

\draw[rotate=\t*\th]
	(90:\r) -- (90+\th:\r)
	(90:3*\r) -- (90+\th:3*\r)
;

\draw[line width=1.75, rotate=\t*\th]
	(90:\r) -- (90:3*\r) 
;

\node at (90+0.15*\th+\t*\th: 1.1*\r) {\scalebox{0.5}{$\diamond$}};
\node at (90-0.075*\th+\t*\th: 2*\r) {\scalebox{0.5}{$\diamond$}};

}

\end{scope} 

\begin{scope}[xshift=2.75*\s cm, yshift=0.35*\s cm] 

\tikzmath{
\r=0.55;
\th=360/4;
\sh=\r;
\ss=1.75;
}

\fill[gray!40, scale=1]
	(\ss*\sh,\ss*\sh) -- (\ss*\sh,-\ss*\sh) -- (-\ss*\sh,-\ss*\sh) -- (-\ss*\sh,\ss*\sh) -- cycle
;

\fill[white]
	(0.5*\th:2*\r) -- (1.5*\th:2*\r) -- (2.5*\th:2*\r)  -- (3.5*\th:2*\r) -- cycle
;

\fill[gray!40]
	(0.5*\th:\r) -- (1.5*\th:\r) -- (2.5*\th:\r)  -- (3.5*\th:\r) -- cycle
;

\foreach \a in {0,...,3} {

\draw[rotate=\th*\a]
	(0.5*\th:\r) -- (1.5*\th:\r)
	(0.5*\th:\r) -- (0.5*\th:\r+\r)
	(0.5*\th:\r+\r) -- (1.5*\th:\r+\r)
;

\draw[rotate=\th*\a, line width=1.75]
	(0.5*\th:\r) -- (0.5*\th:2*\r) 
;

\node at (0.5*\th+0.15*\th+\a*\th: 1.1*\r) {\scalebox{0.5}{$\diamond$}};
\node at (0.5*\th-0.1*\th+\a*\th: 1.5*\r) {\scalebox{0.5}{$\diamond$}};

}

\end{scope}

\begin{scope}[xshift=5.75*\s cm, yshift=0.35*\s cm] 

\tikzmath{
\r=0.5;
\g=5;
\ph=360/\g;
\e=2.35;
}

\fill[gray!40]
	(90:\e*\r) -- (90+1*\ph:\e*\r) -- (90+2*\ph:\e*\r) -- (90+3*\ph:\e*\r) -- (90+4*\ph:\e*\r) -- cycle
;

\fill[white]
	(90:2*\r) -- (90+1*\ph:2*\r) -- (90+2*\ph:2*\r) -- (90+3*\ph:2*\r) -- (90+4*\ph:2*\r) -- cycle
;

\fill[gray!40]
	(90:\r) -- (90+1*\ph:\r) -- (90+2*\ph:\r) -- (90+3*\ph:\r) -- (90+4*\ph:\r) -- cycle
;

\foreach \p in {0,...,4} {

\draw[rotate=\p*\ph]
	(90-1*\ph:\r) -- (90:\r)
	(90:\r) -- (90:2*\r)
	(90-1*\ph:2*\r) -- (90:2*\r)
;

\draw[rotate=\p*\ph, line width=1.75]
	(90:\r) -- (90:2*\r) 
;

\node at (90+0.15*\ph+\p*\ph: 1.1*\r) {\scalebox{0.5}{$\diamond$}};
\node at (90-0.12*\ph+\p*\ph: 1.6*\r) {\scalebox{0.5}{$\diamond$}};

}

\end{scope}

\node at (7.85*\s,0.5*\s) {\Large $\cdots$};

\begin{scope}[xshift=10*\s cm, yshift=0.4*\s cm] 

\tikzmath{
\r=0.5;
\ps=360/7;
}

\filldraw[gray!40] (0,0) circle (2.325*\r);

\filldraw[white] (0,0) circle (2*\r);

\filldraw[gray!40] (0,0) circle (\r);

\foreach \a in {0,1,2,5,6} {

\draw[rotate=\a*\ps]
	(90:\r) -- (90:2*\r)
;

\draw[rotate=\a*\ps, line width=1.75]
	(90:\r) -- (90:2*\r) 
;

\node at (90+0.2*\ps+\a*\ps: 1.18*\r) {\scalebox{0.5}{$\diamond$}};
\node at (90-0.12*\ps+\a*\ps: 1.8*\r) {\scalebox{0.5}{$\diamond$}};

}

\draw (0,0) circle (\r);

\draw (0,0) circle (2*\r);

\node at (270:1.5*\r) {\small $\cdots$};

\end{scope}

\end{tikzpicture}
\caption{The infinite family of tilings of prism type, $\diamond=\beta$}
\label{Fig-a3/am-abab-Tilings}
\end{figure}
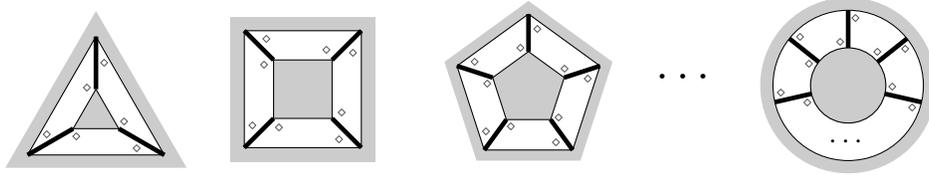


\begin{figure}[h!]
\centering
\begin{tikzpicture}[]

\tikzmath{
\s=1;
}

\begin{scope}[] 

\tikzmath{
\r=0.5;
\th=360/4;
\sh=\r;
\ss=1.75;
}

\fill[gray!40]
	(\ss*\sh,\ss*\sh) -- (\ss*\sh,-\ss*\sh) -- (-\ss*\sh,-\ss*\sh) -- (-\ss*\sh,\ss*\sh) -- cycle
;

\fill[white]
	(0.5*\th:2*\r) -- (1.5*\th:2*\r) -- (2.5*\th:2*\r)  -- (3.5*\th:2*\r) -- cycle
;

\fill[gray!40]
	(0.5*\th:\r) -- (1.5*\th:\r) -- (2.5*\th:\r)  -- (3.5*\th:\r) -- cycle
;

\foreach \a in {0,...,3} {

\draw[rotate=\th*\a]
	(0.5*\th:\r) -- (1.5*\th:\r)
	(0.5*\th:\r) -- (0.5*\th:\r+\r)
	(0.5*\th:\r+\r) -- (1.5*\th:\r+\r)
;

\draw[rotate=\th*\a, line width=1.75]
	(0.5*\th:\r) -- (0.5*\th:2*\r) 
;

}

\draw[]
	(0.5*\th:\r) -- (2.5*\th:\r)
	(1.5*\th:2*\r) -- (1.5*\th:2.5*\r)
	(3.5*\th:2*\r) -- (3.5*\th:2.5*\r)
;

\draw[dashed]
	(1.5*\th:2*\r) -- (1.5*\th:3*\r)
	(3.5*\th:2*\r) -- (3.5*\th:3*\r)
;

\foreach \a in {0,1} {

\node at (1.5*\th+0.12*\th+\a*2*\th: 1.15*\r) {\scalebox{0.5}{$\diamond$}};
\node at (1.5*\th-0.12*\th+\a*2*\th: 1.15*\r) {\scalebox{0.5}{$\diamond$}};

\node at (0.5*\th+0.085*\th+\a*2*\th: 1.55*\r) {\scalebox{0.5}{$\diamond$}};
\node at (0.5*\th-0.085*\th+\a*2*\th: 1.55*\r) {\scalebox{0.5}{$\diamond$}};

}

\end{scope}

\begin{scope}[xshift=2.5*\s cm] 

\tikzmath{
\r=0.4;
\th=360/4;
\x = \r*cos(0.5*\th);
\ph=360/5;
\sh=\r;
\ss=2;
\e=1.35;
}

\fill[gray!40]
	(90:\e*2*\r) -- (90-\ph:\e*2*\r) -- (90-2*\ph:\e*2*\r) -- (90-3*\ph:\e*2*\r) -- (90-4*\ph:\e*2*\r) --  cycle
;

\fill[white]
	(90:2*\r) -- (90-\ph:2*\r) -- (90-2*\ph:2*\r) -- (90-3*\ph:2*\r) -- (90-4*\ph:2*\r) --  cycle
;

\fill[gray!40]
	(90:\r) -- (90-\ph:\r) -- (90-2*\ph:\r) -- (90-3*\ph:\r) -- (90-4*\ph:\r) --  cycle
;

\foreach \a in {0,...,4} {

\draw[rotate=\ph*\a]
	(0,0) -- (90-1*\ph:\r)
	(90:\r) -- (90-1*\ph:\r)
	(90:\r) -- (90:2*\r)
	(90:2*\r) -- (90-1*\ph: 2*\r)
	(90:2*\r) -- (90:2.75*\r)
;

\draw[rotate=\ph*\a, line width=1.75]
	(90-1*\ph:\r) -- (90-1*\ph:2*\r) 
;

\node at (90+0.2*\ph+\a*\ph: 1.1*\r) {\scalebox{0.5}{$\diamond$}};
\node at (90-0.15*\ph+\a*\ph: 1.6*\r) {\scalebox{0.5}{$\diamond$}};

}

\end{scope}

\begin{scope}[xshift=5.25*\s cm] 

\tikzmath{
\r=0.4;
\g=6;
\ph=360/\g;
\x=\r*cos(\ph/2);
\y=\r*sin(\ph/2);
\rr=2*\y/sqrt(2);
\h=3;
\th=360/\h;
\xx=\rr*cos(\th/2);
\R=2*\r/cos(60);
\XX=sqrt((\r)^2+(\r)^2-2*(\r)*(\r)*cos(60));
\X=sqrt((2*\r)^2+(2*\r)^2-2*(2*\r)*(2*\r)*cos(60));
}

\filldraw[gray!40] (0,0) circle (2.75*\r);

\fill[white]
	(90:2*\r) -- (90+\ph:2*\r) -- (90+2*\ph:2*\r) -- (90+3*\ph:2*\r)-- (90+4*\ph:2*\r) -- (90+5*\ph:2*\r) -- cycle
;

\fill[gray!40]
	(90:\r) -- (90+\ph:\r) -- (90+2*\ph:\r) -- (90+3*\ph:\r)-- (90+4*\ph:\r) -- (90+5*\ph:\r) -- cycle
;

\draw[]
	(90-\ph:\r) -- (90+\ph:\r)
	(90-\ph:\r) -- (90+2*\ph:\r)
	(90-2*\ph:\r) -- (90+2*\ph:\r)
;

\arcThroughThreePoints[]{$(90-2*\ph:2*\r)$}{$(90-\ph:2.5*\r)$}{$(90:2*\r)$};
\arcThroughThreePoints[]{$(90+\ph:2*\r)$}{$(90+2*\ph:2.5*\r)$}{$(90+3*\ph:2*\r)$};

\foreach \a in {0,1} {

\draw[rotate=\a*180]
	(90:2*\r) -- ([shift={(90:2*\r)}]135:0.5*\r)
;

\draw[dashed, rotate=\a*180]
	([shift={(90:2*\r)}]135:0.5*\r) -- ([shift={(90:2*\r)}]135:1.25*\r)
;

}

\foreach \p in {0,...,5} {

\draw[rotate=\p*\ph]
	(90-1*\ph:\r) -- (90:\r)
	(90:\r) -- (90:2*\r)
	(90-1*\ph:2*\r) -- (90:2*\r)
;

\draw[line width=2, rotate=\p*\ph]
	(90:\r) -- (90:2*\r)
;

}

\foreach \a in {0,1} {

\node at (90+0.225*\ph+\a*180: 1.1*\r) {\scalebox{0.5}{$\diamond$}};
\node at (90-0.225*\ph+\a*180: 1.1*\r) {\scalebox{0.5}{$\diamond$}};

\node at (90-0.85*\ph+\a*180: 1.6*\r) {\scalebox{0.5}{$\diamond$}};
\node at (90-1.15*\ph+\a*180: 1.6*\r) {\scalebox{0.5}{$\diamond$}};

\node at (90+0.85*\ph+\a*180: 1.6*\r) {\scalebox{0.5}{$\diamond$}};
\node at (90+1.225*\ph+\a*180: 1.1*\r) {\scalebox{0.5}{$\diamond$}};

}

\end{scope} 

\begin{scope}[xshift=8.1*\s cm] 

\tikzmath{
\r=0.4;
\g=6;
\ph=360/\g;
\x=\r*cos(\ph/2);
\y=\r*sin(\ph/2);
\rr=2*\y/sqrt(2);
\h=3;
\th=360/\h;
\xx=\rr*cos(\th/2);
\R=2*\r/cos(60);
\XX=sqrt((\r)^2+(\r)^2-2*(\r)*(\r)*cos(60));
\X=sqrt((2*\r)^2+(2*\r)^2-2*(2*\r)*(2*\r)*cos(60));
}

\filldraw[gray!40] (0,0) circle (2.75*\r);

\fill[white]
	(90:2*\r) -- (90+\ph:2*\r) -- (90+2*\ph:2*\r) -- (90+3*\ph:2*\r)-- (90+4*\ph:2*\r) -- (90+5*\ph:2*\r) -- cycle
;

\fill[gray!40]
	(90:\r) -- (90+\ph:\r) -- (90+2*\ph:\r) -- (90+3*\ph:\r)-- (90+4*\ph:\r) -- (90+5*\ph:\r) -- cycle
;

\foreach \t in {0,...,2} {

\draw[rotate=\t*\th]
	(90-\ph:\r) -- (90+\ph:\r)
;

}

\arcThroughThreePoints[]{$(90-2*\ph:2*\r)$}{$(90-\ph:2.5*\r)$}{$(90:2*\r)$};
\arcThroughThreePoints[]{$(90:2*\r)$}{$(90+\ph:2.5*\r)$}{$(90+2*\ph:2*\r)$};
\arcThroughThreePoints[]{$(90+2*\ph:2*\r)$}{$(90-3*\ph:2.5*\r)$}{$(90+4*\ph:2*\r)$};

\foreach \p in {0,...,5} {

\draw[rotate=\p*\ph]
	(90-1*\ph:\r) -- (90:\r)
	(90:\r) -- (90:2*\r)
	(90-1*\ph:2*\r) -- (90:2*\r)
;

\draw[line width=2, rotate=\p*\ph]
	(90:\r) -- (90:2*\r)
;

}

\foreach \a in {0,1,2} {

\node at (90+0.225*\ph+\a*2*\ph: 1.1*\r) {\scalebox{0.5}{$\diamond$}};
\node at (90-0.225*\ph+\a*2*\ph: 1.1*\r) {\scalebox{0.5}{$\diamond$}};

\node at (90+1.15*\ph+\a*2*\ph: 1.6*\r) {\scalebox{0.5}{$\diamond$}};
\node at (90+0.85*\ph+\a*2*\ph: 1.6*\r) {\scalebox{0.5}{$\diamond$}};

}

\end{scope} 

\begin{scope}[xshift=11*\s cm] 

\tikzmath{
\r=0.25;
\g=6;
\ph=360/\g;
\x=\r*cos(\ph/2);
\y=\r*sin(\ph/2);
\rr=2*\y/sqrt(2);
\h=3;
\th=360/\h;
\xx=\rr*cos(\th/2);
\R=2*\r/cos(60);
}

\filldraw[gray!40] (0,0) circle (1.1*\R);

\filldraw[white] (0,0) circle (3*\r);

\filldraw[gray!40]
	(90:2*\r) -- (90+\ph:2*\r) --(90+2*\ph:2*\r) --(90+3*\ph:2*\r) -- (90+4*\ph:2*\r) -- (90+5*\ph:2*\r) -- cycle
;

\foreach \t in {0,...,2} {

\draw[rotate=\t*\th]
	(90-\ph:0.75*\r) -- (90+\ph:0.75*\r)
	(90-\ph:0.75*\r) -- (90-\ph:2*\r) 
	(90-\ph:0.75*\r) -- (90:2*\r)
	(90+\ph:0.75*\r) -- (90:2*\r)
;

\draw[rotate=\t*\th]
	(90:3*\r) -- (90:\R)
	(90:\R) to[out=200, in=80] (90+\ph:3*\r) 
	(90:\R) to[out=-20, in=110] (90-\ph:3*\r) 
	(90:\R) arc (90:90+\th:\R)
;

\node at (90+0.075*\th+\t*\th: 2.6*\r) {\scalebox{0.5}{$\diamond$}};
\node at (90-0.075*\th+\t*\th: 2.6*\r) {\scalebox{0.5}{$\diamond$}};

\node at (90+0.8*\ph+\t*\th: 2.1*\r) {\scalebox{0.5}{$\diamond$}};
\node at (90+1.2*\ph+\t*\th: 2.1*\r) {\scalebox{0.5}{$\diamond$}};

}

\foreach \p in {0,...,5} {

\draw[rotate=\p*\ph]
	(90-1*\ph:2*\r) -- (90:2*\r)
	(90:2*\r) -- (90:3*\r)
	(90:3*\r) arc (90:90+\ph:3*\r)
;

\draw[line width=2, rotate=\p*\ph]
	(90:2*\r) -- (90:3*\r)
;

}

\end{scope}

\end{tikzpicture}
\caption{The sporadic tilings, $\diamond=\beta$}
\label{Fig-a3-abab-Sporadic-Tilings}
\end{figure}
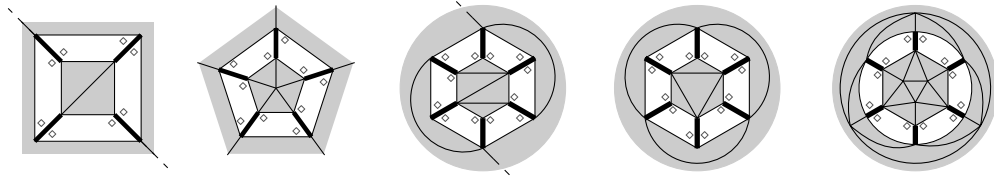

The paper is organised as follows. In Section \ref{Sec-basic}, we explain the basic terminologies and tools. In Section \ref{Sec-mgon-tilings}, we classify the tilings by regular $m$-gons ($m\ge4$) and the quadrilaterals. In Section \ref{Sec-tri-tilings}, we classify the tilings by regular triangles and the quadrilaterals. The tilings in Figure \ref{Fig-a3/am-abab-Tilings} are obtained in Propositions \ref{Prop-am-abab-albega}, \ref{Prop-a3-abab-albega}. The tilings in Figure \ref{Fig-a3-abab-Sporadic-Tilings} are obtained in Propositions \ref{Prop-a3-abab-albega}, \ref{Prop-a3-abab-albe2}, \ref{Prop-a3-abab-al2be2}. 

\section{Basics} \label{Sec-basic}

We denote a vertex by $\alpha^a\beta^b\gamma^c$, which consists of $a$ copies of $\alpha$ and $b$ copies of $\beta$ and $c$ copies of $\gamma$. The {\em vertex angle sum} of a vertex is given by
\begin{align}\label{Eq-vertex-angle-sum}
a \alpha + b \beta + c \gamma = 2\pi.
\end{align} 
In a vertex notation, $a,b,c$ are assumed to be $>0$ unless otherwise specified. That is, we only express the angles appearing at a vertex whenever possible. For example, $\alpha\beta^2$ is a vertex with $a=1$ and $b=2$ and $c=0$. The notation $\alpha\beta^2\cdots$ means a vertex with at least one $\alpha$ and two $\beta$'s, i.e., $a\ge1$ and $b\ge2$. The angle combination in $\cdots$ is called the {\em remainder} of the vertex. 

To obtain the vertices, it is convenient to have notations for studying various angle arrangements. For example, $\alpha_1\gamma_2\cdots$ denotes the vertex where $T_1$ contributes $\alpha$ and $T_2$ contributes $\gamma$ in the first picture of Figure \ref{Fig-a4-abab-adj-square-quad}. To emphasize $\alpha_1$ being adjacent to $\gamma_2$ along an $x$-edge \quotes{ $\vert$ }, we use $\alpha_1 \vert \gamma_2 \cdots$ to denote the vertex. We use \quotes{ $\bvert$ } to denote the $y$-edge. In addition, the same picture shows that $\alpha\vert\gamma$ is a vertex if and only if $\alpha\vert\beta\cdots$ is a vertex. Similarly, $T_1, T_2$ in the second picture show that $\gamma\vert\gamma\cdots$ is a vertex if and only if $\beta\vert\beta\cdots$ is also a vertex. 

\begin{figure}[h!] 
\centering
\begin{tikzpicture}

\tikzmath{
\s=1;
\r=0.8;
\th=360/4;
\x=\r*cos(0.5*\th);
\R=sqrt(\x^2+(3*\x)^2);
\aR=acos(3*\x/\R);
\ph=360/5;
\y=\r*sin(\ph/2);
\rr=2*\y/sqrt(2);
\xx=\rr*cos(\th/2);
\X=\r*cos(0.5*\ph);
}

\begin{scope}[] 

\foreach \a in {0,1,4} {

\draw[rotate=\a*\ph]
	(0.5*\ph:\r) -- (-0.5*\ph:\r)
;

}

\foreach \a in {0,...,3} {

\tikzset{shift={(\X+\xx,0)}}

\draw[rotate=\a*\th]
	(1.5*\th:\rr) -- (0.5*\th:\rr)
;

}

\draw[line width=2, shift={(\X+\xx,0)}]
	(1.5*\th:\rr) -- (0.5*\th:\rr)
	(2.5*\th:\rr) -- (3.5*\th:\rr)
;

\foreach \a in {0,1} {

\tikzset{shift={(\X+\xx,0)}}

\node at (1.55*\th+\a*2*\th:0.6*\rr) {\small $\beta$};
\node at (0.5*\th+\a*2*\th:0.6*\rr) {\small $\gamma$};

}

\foreach \a in {0,4} {

\node at (0.5*\ph+\a*\ph:0.75*\r) {\small $\alpha$};

}

\node at (160:0.45*\r) {$\vdots$};

\node[inner sep=1,draw,shape=circle] at (0.15*\r,0) {\small $1$};
\node[inner sep=1,draw,shape=circle] at (\X+\xx,0) {\small $2$};

\end{scope} 

\begin{scope}[xshift=4*\s cm] 

\foreach \aa in {-1,1} {

\tikzset{shift={(\aa*\x,0)}, xscale=\aa}

\foreach \a in {0,...,3} {

\draw[rotate=\th*\a]
	(0.5*\th:\r) -- (1.5*\th:\r)
;

}

\foreach \a in {0,2} {

\draw[line width=2]
	(90-0.5*\th:\r)  -- (90+0.5*\th:\r)
	(90+1.5*\th:\r)  -- (90+2.5*\th:\r)
;

\node at (1.5*\th+\th*\a: 0.6*\r) {\small $\beta$}; 
\node at (0.5*\th+\th*\a: 0.6*\r) {\small $\gamma$}; 

}

}

\node[inner sep=1,draw,shape=circle] at (-\x,0) {\small $1$};
\node[inner sep=1,draw,shape=circle] at (\x,0) {\small $2$};

\end{scope} 

\begin{scope}[xshift=7*\s cm, yshift=-0.1*\s cm] 

\foreach \a in {0,1,2} {

\draw[rotate=120*\a]
	(0:0) -- (90:\r) 
;

\node at (270+\a*120:0.3*\r) {\small $\alpha$};

}

\end{scope} 

\begin{scope}[xshift=9.5*\s cm, yshift=-0.1*\s cm] 

\foreach \a in {0,1,2} {

\draw[rotate=120*\a]
	(0:0) -- (90:\r) 
;

}

\draw[line width=2]
	(0:0) -- (90:\r) 
;

\node at (270:0.3*\r) {\small $\alpha$};

\node at (90-60:0.3*\r) {\small $\beta$};

\node at (90+60:0.3*\r) {\small $\gamma$};

\end{scope} 

\end{tikzpicture}
\caption{The arrangements of $\alpha \vert \gamma$ and $\gamma\vert\gamma$ and $\beta\vert\beta$ and $\alpha^3, \alpha\beta\gamma$}
\label{Fig-a4-abab-adj-square-quad}
\end{figure}
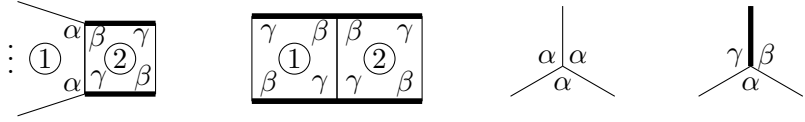

For a full vertex, such as $\alpha^3$ in the third picture of Figure \ref{Fig-a4-abab-adj-square-quad}, we use $\vert \alpha \vert \alpha \vert \alpha \vert$ to denote its angle arrangement. Similarly, we use $\vert \alpha \vert \beta \, \bvert \, \gamma \vert$ to denote the angle arrangement of $\alpha\beta\gamma$ in the fourth picture.

Up to symmetry, we may assume $\beta > \gamma$ in the quadrilateral. This assumption is implicit throughout this paper.

The prototiles in Figure \ref{Fig-a4-abab-angles} are regular $m$-gons ($m\ge3$) with angles $\alpha$ and the quadrilateral with angles $\beta,\gamma$. We have $m\alpha>(m-2)\pi$ and $2\beta+2\gamma>2\pi$. Combined with $\beta>\gamma$, we get
\begin{align}\label{Eq-am-abab-angle-ineq}
\alpha > (1 - \tfrac{2}{m})\pi, \quad 
\beta > \tfrac{1}{2}\pi, \quad
\beta+\gamma > \pi.
\end{align}
We have $\alpha>\frac{1}{3}\pi$ for $m=3$ and $\alpha>\frac{1}{2}\pi$ for $m\ge4$. Moreover, $\alpha^a=\alpha^3, \alpha^4, \alpha^5$ for $m=3$ and $\alpha^a=\alpha^3$ for $m=4,5$ and $\alpha^a$ is not a vertex for $m>6$.

The following lemma is an adaptation of \cite[Lemma 2]{cly}.

\begin{lem}[Parity Lemma] \label{Lem-Parity} The total number of $\beta,\gamma$ at each vertex is even.
\end{lem}

\begin{proof} The total number of $\beta,\gamma$ at a vertex is twice the number of $y$-edges at the vertex.
\end{proof}

The next lemma is an adaptation of \cite[Lemma 4]{cly} and an immediate consequence of Parity Lemma (Lemma \ref{Lem-Parity}).

\begin{lem}[Counting Lemma] \label{Lem-Counting} If at every vertex the number of $\beta$ is no more than the number of $\gamma$, then at every vertex these two numbers are equal. Moreover, $\beta^2\cdots$ is a vertex if and only if $\gamma^2\cdots$ is vertex.
\end{lem}

The first picture of Figure \ref{Fig-a4-abab-adj-square-quad} shows that $\alpha\beta\cdots, \alpha\gamma\cdots$ are vertices. This gives the following lemma.

\begin{lem}\label{Lem-albe-alga} In a dihedral tiling by the regular polygons and the quadrilaterals, both $\alpha\beta\cdots, \alpha\gamma\cdots$ are vertices.
\end{lem}

By Parity Lemma, we also know $\alpha\beta\cdots=\alpha\beta^2\cdots, \alpha\beta\gamma\cdots$ and $\alpha\gamma\cdots=\alpha\beta\gamma\cdots, \alpha\gamma^2\cdots$. Then Lemma \ref{Lem-albe-alga} implies that one of $\alpha\beta^2\cdots, \alpha\beta\gamma\cdots$ is a vertex. By $\beta>\gamma$, it implies $\alpha+\beta+\gamma\le2\pi$. By $\beta+\gamma>\pi$, we then have $\alpha<\pi$. By $\beta>\gamma$, we then have $\gamma<\pi$.

By $\beta>\gamma$ and $\beta+\gamma > \pi$ and Parity Lemma, a vertex $\beta^2\cdots$ has no more $\beta,\gamma$ in the remainder. So we get
\begin{align}
\label{Eq-a3-abab-be2}
m = 3: \quad &\beta^2\cdots=\alpha\beta^2, \alpha^2\beta^2; \\
\label{Eq-am-abab-be2}
m \ge 4: \quad &\beta^2\cdots=\alpha\beta^2.
\end{align}
From the above, we know $\beta^2\cdots$ is $\alpha\beta^2=$ $\bvert \, \beta \vert \alpha \vert \beta \, \bvert$ or $ \alpha^2\beta^2=$ $\bvert \, \beta \vert \alpha \vert \alpha \vert \beta \, \bvert$. Then $\beta \vert \beta \cdots$ is not a vertex. By the second picture of Figure \ref{Fig-a4-abab-adj-square-quad}, this further implies that $\gamma\vert\gamma\cdots$ is not a vertex. By no $\gamma\vert\gamma\cdots$, we know that $\gamma^c, \beta\gamma^{c}$ are not vertices and $\alpha\gamma^c=\alpha\gamma^2$. By $\beta>\gamma$, the vertex $\alpha\gamma^2$ contradicts $\alpha+\beta+\gamma\le2\pi$. Hence, by \eqref{Eq-a3-abab-be2}, \eqref{Eq-am-abab-be2}, we have 
\begin{align} \label{Eq-a3/am-abab-ga2}
\gamma^2\cdots=\alpha^{a\ge2}\gamma^c, \alpha^a\beta\gamma^c.
\end{align}

The $y$-edges divide a vertex into a combination of $\bvert \, \beta \vert \cdots \vert \beta \, \bvert$, $\bvert \, \beta \vert \cdots \vert \gamma \, \bvert$ and $\bvert \, \gamma \vert \cdots \vert \gamma \, \bvert$, where $\cdots$ is empty or filled by $\alpha$'s. By no $\gamma\vert\gamma\cdots$, we know $\bvert \, \gamma \vert \cdots \vert \gamma \, \bvert$ has at least one $\alpha$. This implies $2a \ge c-1$ in $\alpha^a\beta\gamma^c$ and $2a \ge c$ in $\alpha^a\gamma^c$.


By $\beta^2\cdots=\alpha\beta^2, \alpha^2\beta^2$ and no $\beta\gamma^c$, we get $\beta\gamma\cdots=\alpha^a\beta\gamma^c$. By $\beta+\gamma>\pi$ and the first inequality in \eqref{Eq-am-abab-angle-ineq} and $2a \ge c$ in $\alpha^a\beta\gamma^c$, we have $a\le2$ for $m=3$ and $a=1$ for $m\ge4$. By Parity Lemma, we get
\begin{align}
\label{Eq-a3-abab-bega}
m = 3: \quad &\beta\gamma\cdots=\alpha\beta\gamma, \alpha^2\beta\gamma, \alpha\beta\gamma^3, \alpha^2\beta\gamma^3, \alpha^2\beta\gamma^5; \\
\label{Eq-am-abab-bega}
m \ge 4: \quad &\beta\gamma\cdots=\alpha\beta\gamma, \alpha\beta\gamma^3.
\end{align}

We recall a well-known fact below.

\begin{lem}\label{Lem-deg345} In a tiling of the sphere by polygons, there is a degree $3, 4$ or $5$ vertex. Moreover, if there is no triangle, then there is a degree $3$ vertex.
\end{lem}

By \eqref{Eq-am-abab-angle-ineq} and Lemma \ref{Lem-deg345} and Parity Lemma and no $\gamma^c, \alpha\gamma^c, \beta\gamma^c$, one of the following is a vertex in a dihedral tiling
\begin{align}
\label{Eq-List-a3-abab-deg345}
&m=3:& &\alpha^3,  \alpha\beta^2, \alpha\beta\gamma, \alpha^4, \alpha^2\beta^2, \alpha^2\gamma^2,  \alpha^2\beta\gamma, \alpha^5,  \alpha^3\gamma^2, \alpha\beta\gamma^3; \\
\label{Eq-List-a4-a5-abab-deg3}
&m=4,5:& &\alpha^3, \alpha\beta^2, \alpha\beta\gamma; \\
\label{Eq-List-a6+-abab-deg3}
&m \ge 6:& &\alpha\beta^2, \alpha\beta\gamma.
\end{align}

\section{Tilings with $m$-gons with $m\ge4$} \label{Sec-mgon-tilings}



By \eqref{Eq-am-abab-be2}, we have $\beta^2\cdots=\alpha\beta^2$. We divide the discussion according to whether $\alpha\beta^2$ is a vertex. We also recall $\alpha^a=\alpha^3$.

\begin{prop}\label{Prop-am-abab-albega} The dihedral tiling without $\alpha\beta^2$ is the second picture of Figure \ref{Fig-am-abab-AAD-Tiling-al3-albega}.
\end{prop}

The tiling is given by the cube and has $2$ regular polygons and $4$ quadrilaterals.

\begin{proof} By \eqref{Eq-am-abab-be2} and no $\alpha\beta^2$, we know that $\beta^2\cdots$ is not a vertex. Then by Counting Lemma, $\gamma^2\cdots$ is also not a vertex. Parity Lemma further implies $\beta\cdots=\gamma\cdots=\beta\gamma\cdots$ with no $\beta,\gamma$ in the remainder. By \eqref{Eq-am-abab-bega}, we then know $\beta\cdots=\gamma\cdots=\alpha\beta\gamma$ is a vertex. So we have $\beta\cdots=\gamma\cdots=\beta\gamma\cdots=\alpha\beta\gamma$. The other vertices consist of only $\alpha$'s. By $\alpha^a=\alpha^3$, the vertices are 
\begin{align*}
\alpha^3, \alpha\beta\gamma.
\end{align*}

From the above, we know $\alpha^2\cdots=\alpha^3$. Starting at an $\alpha^3$, namely $\alpha_1\alpha_2\alpha_3$ in the first picture of Figure \ref{Fig-am-abab-AAD-Tiling-al3-albega}, we also determine its two adjacent vertices to be $\alpha^3$'s. Repeating the same argument, we always get the vertex $\alpha^3$. Then the tiling is a monohedral tiling and therefore $m=4,5$ and the monohedral tilings are the cube and the dodecahedron. Hence $\alpha^3$ is not a vertex for dihedral tiling.

\begin{figure}[h!] 
\centering
\begin{tikzpicture}

\tikzmath{
\s=1;
\r=0.8;
\g=6;
\ph=360/\g;
\x=\r*cos(\ph/2);
\y=\r*sin(\ph/2);
\rr=2*\y/sqrt(2);
\h=4;
\th=360/\h;
\xx=\rr*cos(\th/2);
\ps=360/7;
}

\begin{scope}[yshift=1*\s cm] 

\foreach \p in {2,3,4,5} {

\draw[rotate=\p*\ph]
	(90-1*\ph:\r) -- (90:\r)
;

}

\foreach \p in {-1,0,1} {

\draw[rotate=\p*\ph]
	(270:\r) -- (270:2*\r)
;

}

\foreach \p in {-1,0,1} {

\node at (282.5+\p*\ph:1.15*\r) {\small $\alpha$};
\node at (257.5+\p*\ph:1.15*\r) {\small $\alpha$};

}

\node at (270:0.75*\r) {\small $\alpha$};
\node at (270-\ph:0.7*\r) {\small $\alpha$};
\node at (270+\ph:0.7*\r) {\small $\alpha$};

\node[inner sep=1,draw,shape=circle] at (-0.8*\r,-1.5*\r) {\small $1$};
\node[inner sep=1,draw,shape=circle] at (0.8*\r,-1.5*\r) {\small $2$};
\node[inner sep=1,draw,shape=circle] at (0,0) {\small $3$};

\end{scope}

\begin{scope}[xshift=4*\s cm]

\foreach \p in {-1,0,1,2} {

\draw[rotate=\p*\ph]
	(90-1*\ph:\r) -- (90:\r)
;

}

\foreach \p in {-1,0,1} {

\draw[line width=1.75, rotate=\p*\ph]
	(90:\r) -- (90:2*\r)
;

\node at (78.5+\p*\ph:1.15*\r) {\small $\beta$};
\node at (102.5+\p*\ph:1.15*\r) {\small $\gamma$};

}

\node at (90:0.75*\r) {\small $\alpha$};
\node at (90-\ph:0.7*\r) {\small $\alpha$};
\node at (90+\ph:0.7*\r) {\small $\alpha$};

\foreach \aa in {-1, 1} {

\draw[xscale=\aa]
	(90:2*\r) -- (90-1*\ph:2*\r) 
;

}

\node at (85:1.65*\r) {\small $\gamma$};
\node at (97.5:1.65*\r) {\small $\beta$};

\node at (90-0.85*\ph:1.6*\r) {\small $\beta$};
\node at (90+0.85*\ph:1.6*\r) {\small $\gamma$};

\node at (90:2.25*\r) {\small $\alpha$};

\node at (90-1.15*\ph:1.7*\r) {\small $\gamma$};
\node at (90+1.175*\ph:1.7*\r) {\small $\beta$};

\node[inner sep=1,draw,shape=circle] at (0,0) {\small $1$};
\node[inner sep=1,draw,shape=circle] at (0.65*\r,1.2*\r) {\small $2$};
\node[inner sep=1,draw,shape=circle] at (-0.65*\r,1.2*\r) {\small $3$};

\end{scope}

\begin{scope}[xshift=8.5*\s cm] 

\foreach \a in {0,1,2,5,6} {

\draw[line width=1.75, rotate=\a*\ps]
	(90:\r) -- (90:2*\r)
;

\node at (90+\a*\ps:0.8*\r) {\small $\alpha$};

\node at (90-0.25*\ps+\a*\ps:1.2*\r) {\small $\beta$};
\node at (90+0.2*\ps+\a*\ps:1.2*\r) {\small $\gamma$};

\node at (90-0.15*\ps+\a*\ps:1.75*\r) {\small $\gamma$};
\node at (90+0.2*\ps+\a*\ps:1.75*\r) {\small $\beta$};

\node at (90+\a*\ps:2.2*\r) {\small $\alpha$};

}

\draw (0,0) circle (\r);

\draw (0,0) circle (2*\r);

\node at (270:1.5*\r) {\small $\cdots$};

\end{scope} 

\end{tikzpicture}
\caption{The tiling with $\alpha^3$ and the tiling with $\alpha\beta\gamma$}
\label{Fig-am-abab-AAD-Tiling-al3-albega}
\end{figure}

Now $\alpha\beta\gamma$ is the only vertex. It determines tiles $T_1, T_2, T_3$ in the second picture of Figure \ref{Fig-am-abab-AAD-Tiling-al3-albega}. Then we have $\beta_3\gamma_2\cdots=\alpha\beta\gamma$. Repeating the process, we determine a dihedral tiling for each $m\ge4$. 
\end{proof}


\begin{prop}\label{Prop-am-abab-albe2} There is no dihedral tiling with $\alpha\beta^2$.
\end{prop}

\begin{proof} By $\beta+\gamma>\pi$ and $\alpha\beta^2$, we get $2\gamma>\alpha$. Then we have $\alpha+\beta+3\gamma>2\pi$. By $\beta>\gamma$ and $\alpha\beta^2$, we know that $\alpha\beta\gamma$ is not a vertex. Combined with \eqref{Eq-am-abab-bega} and $\alpha+\beta+3\gamma>2\pi$, we know that $\beta\gamma\cdots$ is not a vertex.  

By $\alpha\beta^2$ and Counting Lemma, we know that $\gamma^2\cdots$ is a vertex. By \eqref{Eq-a3/am-abab-ga2}, we have $\gamma^2\cdots=\alpha^{a\ge2}\gamma^c, \alpha^a\beta\gamma^c$. By Parity Lemma and no $\beta\gamma\cdots$, we further know that $\gamma\cdots=\gamma^2\cdots=\alpha^{a\ge2}\gamma^{c}$ is a vertex. Then by $2\gamma>\alpha$ and Parity Lemma, we have $2\pi \ge 2\alpha+2\gamma>3\alpha$, which implies $\alpha<\frac{2}{3}\pi$. By $\alpha\beta^2$, we get $\beta>\frac{2}{3}\pi>\alpha$.

By $2\gamma> \alpha > \frac{1}{2}\pi$ and Parity Lemma, we conclude $\gamma\cdots=\alpha^{a\ge2}\gamma^{c}=\alpha^{2}\gamma^2$. By $\beta > \alpha$ and \eqref{Eq-am-abab-be2} and Parity Lemma and no $\beta\gamma\cdots$, we get $\beta\cdots=\beta^2\cdots=\alpha\beta^2$. Meanwhile, by $\beta>\alpha$ and $\alpha\beta^2$, we know that $\alpha^a=\alpha^3$ is not a vertex. Hence we get
\begin{align*}
\alpha\beta^2, \alpha^2\gamma^2.
\end{align*}
From the above, we have $\alpha^2\cdots=\alpha^2\gamma^2$ with a unique angle arrangement $\bvert \, \gamma \vert \alpha \vert \alpha \vert \gamma \, \bvert$. 

By $\beta>\alpha$ and $\alpha\beta^2$, we have $\frac{2}{3}\pi>\alpha$. By $m\ge4$ and $\alpha>(1 - \frac{2}{m})\pi$, we get $m=4,5$.

For $m=4$, the arrangement of $\alpha\vert\alpha$ determines tiles $T_1, T_2$ in the first picture of Figure \ref{Fig-a4/a5-abab-AAD-albe2-al2ga2}. Then $\alpha^2\cdots=\alpha^2\gamma^2$ determines $T_3, T_4$. Then $\alpha_2\beta_3\cdots, \alpha_2\beta_4\cdots=\alpha\beta^2$ give two adjacent $\beta$'s in $T_5$, a contradiction. So $\alpha\vert\alpha\cdots$ is not a vertex. This further implies $\alpha^2\gamma^2$ is not a vertex, contradicting Lemma \ref{Lem-albe-alga}.

\begin{figure}[h!] 
\centering
\begin{tikzpicture}

\tikzmath{
\s=1;
}

\begin{scope}[]

\tikzmath{
\r=0.8;
\th=360/4;
\x=\r*cos(0.5*\th);
\R = sqrt(\x^2+(3*\x)^2);
\aR = acos(3*\x/\R);
}

\foreach \aa in {-1,1} {

\tikzset{shift={(\aa*\x,0)}}

\foreach \a in {0,...,3} {

\draw[rotate=\th*\a]
	(0.5*\th:\r) -- (1.5*\th:\r)
;

\node at (0.5*\th+\th*\a: 0.65*\r) {\small $\alpha$}; 

}

}

\foreach \aa in {-1,1} {

\tikzset{shift={(\x,\aa*2*\x)}, yscale=\aa}

\foreach \a in {0,...,3} {

\draw[rotate=\th*\a]
	(0.5*\th:\r) -- (1.5*\th:\r)
;

}

\foreach \a in {0,1} {

\node at (1.5*\th+\a*2*\th: 0.65*\r) {\small $\beta$}; 
\node at (0.5*\th+\a*2*\th: 0.65*\r) {\small $\gamma$}; 

}

}

\node at (-0.35*\x, 1.35*\x) {\small $\gamma$}; 
\node at (-0.35*\x, -1.35*\x) {\small $\gamma$}; 

\draw[line width=2]
	(0,\x) -- (0,3*\x)
	(0,-\x) -- (0,-3*\x)
	(2*\x,\x) -- (2*\x,3*\x)
	(2*\x,-\x) -- (2*\x,-3*\x)
;

\node at (2.25*\x, \x) {\small $\beta$};
\node at (2.25*\x, -\x) {\small $\beta$};

\node[inner sep=1,draw,shape=circle] at (-\x,0) {\small $1$};
\node[inner sep=1,draw,shape=circle] at (\x,0) {\small $2$};
\node[inner sep=1,draw,shape=circle] at (\x,2*\x) {\small $3$};
\node[inner sep=1,draw,shape=circle] at (\x,-2*\x) {\small $4$};
\node[inner sep=1,draw,shape=circle] at (3*\x,0) {\small $5$};

\end{scope}

\begin{scope}[xshift=5*\s cm] 

\tikzmath{
\r=0.8;
\g=5;
\ph=360/\g;
\x=\r*cos(\ph/2);
\y=\r*sin(\ph/2);
\rr=2*\y/sqrt(2);
\h=4;
\th=360/\h;
\xx=\rr*cos(\th/2);
}

\foreach \aa in {-1,1} {

\tikzset{shift={(\aa*\x,0)}, xscale=-\aa}

\foreach \a in {0,...,4} {

\draw[rotate=\a*\ph]
	(0.5*\ph:\r) -- (-0.5*\ph:\r)
;

\node at (0.5*\ph+\a*\ph:0.7*\r) {\small $\alpha$};

}

}

\foreach \a in {-1,1} {

\tikzset{shift={(\x,0)}}

\draw[line width=2, rotate=\a*\ph]
	(0:\r) -- (0:1.5*\r)
;

}

\foreach \a in {0,1} {

\draw[line width=2, rotate=\a*2*\th]
	(0,\y) -- (0, 2*\y)
;

}

\draw[line width=2, shift={(\x,0)}]
	(0:\r) -- (10:1.5*\r)
	(0:\r) -- (-10:1.5*\r)
;

\node at (-0.2*\r, 1.5*\y) {\small $\gamma$};
\node at (-0.2*\r, -1.5*\y) {\small $\gamma$};

\node at (0.2*\r, 1.5*\y) {\small $\gamma$};
\node at (0.95*\r, 2*\y) {\small $\beta$};

\node at (1.4*\r, 1.8*\y) {\small $\beta$};
\node at (2.3*\x, 0.5*\y) {\small $\gamma$};

\node at (1.35*\r, -1.75*\y) {\small $\beta$};
\node at (2.3*\x, -0.5*\y) {\small $\gamma$};

\node at (0.2*\r, -1.5*\y) {\small $\gamma$};
\node at (0.95*\r, -2.1*\y) {\small $\beta$};

\node[inner sep=1,draw,shape=circle] at (-\x,0) {\small $1$};
\node[inner sep=1,draw,shape=circle] at (\x,0) {\small $2$};
\node[inner sep=1,draw,shape=circle] at (0.65*\x,1.75*\x) {\small $3$};
\node[inner sep=1,draw,shape=circle] at (0.65*\x,-1.75*\x) {\small $4$};
\node[inner sep=1,draw,shape=circle] at (2.5*\x,1*\x) {\small $5$};
\node[inner sep=1,draw,shape=circle] at (2.5*\x,-1*\x) {\small $6$};

\end{scope}

\end{tikzpicture}
\caption{The arrangements of $\alpha\vert\alpha$}
\label{Fig-a4/a5-abab-AAD-albe2-al2ga2}
\end{figure}
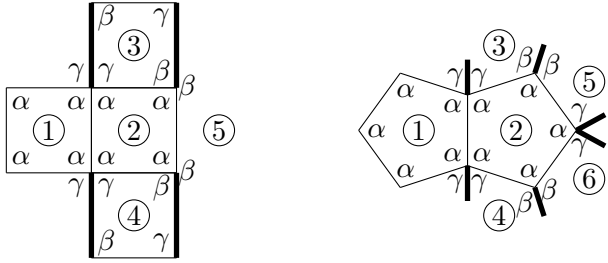

For $m=5$, the arrangement of $\alpha\vert\alpha$ determines tiles $T_1, T_2$ in the second picture of Figure \ref{Fig-a4/a5-abab-AAD-albe2-al2ga2}. Then $\alpha^2\cdots=\alpha^2\gamma^2$ determines $T_3, T_4$. Then $\alpha_2\beta_3\cdots, \alpha_2\beta_4\cdots=\alpha\beta^2$ determine $T_5, T_6$ respectively. It implies $\bvert \, \gamma_5 \vert \alpha_2 \vert \gamma_6 \, \bvert \cdots$ $= \alpha^2\gamma^2$, contradicting $\alpha^2\gamma^2=$ $\bvert \, \gamma \vert \alpha \vert \alpha \vert \gamma \, \bvert$. 
\end{proof}

\section{Tilings with Triangles} \label{Sec-tri-tilings}

By \eqref{Eq-a3-abab-be2}, we have $\beta^2\cdots=\alpha\beta^2,\alpha^2\beta^2$. We divide the discussion according to whether one of $\alpha\beta^2, \alpha^2\beta^2$ is a vertex. We also recall $\alpha^a=\alpha^3, \alpha^4, \alpha^5$.

\begin{prop}\label{Prop-a3-abab-albega} The dihedral tilings without $\alpha\beta^2, \alpha^2\beta^2$ is the first tiling in Figure \ref{Fig-a3/am-abab-Tilings} and the second tiling in Figure \ref{Fig-a3-abab-Sporadic-Tilings}.
\end{prop}

\begin{proof} By \eqref{Eq-a3-abab-be2}, the hypothesis means that $\beta^2\cdots$ is not a vertex. Counting Lemma implies that $\gamma^2\cdots$ is also not a vertex. Parity Lemma further implies $\beta\cdots=\gamma\cdots=\beta\gamma\cdots$ with no $\beta,\gamma$ in the remainder. Then \eqref{Eq-a3-abab-bega} implies that exactly one of $\alpha\beta\gamma, \alpha^2\beta\gamma$ is a vertex.

If $\alpha\beta\gamma$ is a vertex, then the exact same argument in the proof of Proposition \ref{Prop-am-abab-albega} shows that $\alpha\beta\gamma$ is the only vertex of dihedral tilings. Starting at an $\alpha\beta\gamma$, we get the tiling as desired.

If $\alpha^2\beta\gamma$ is a vertex, then \eqref{Eq-a3-abab-bega} implies $\beta\cdots=\gamma\cdots=\beta\gamma\cdots=\alpha^2\beta\gamma$. On the other hand, $\alpha^2\beta\gamma$ and $\beta+\gamma>\pi$ imply $\alpha<\frac{1}{2}\pi$. Then $\alpha^a=\alpha^5$. By no $\alpha\beta^2$, we get 
\begin{align*}
\AVC = \{ \alpha^2\beta\gamma, \alpha^5 \}.
\end{align*}



Starting at an $\alpha^2\beta\gamma$, we determine tiles $T_1, T_2, T_3, T_4$ in the first picture of Figure \ref{Fig-a3-abab-Tiling-al2bega-al5}. Then $\beta_4\gamma_3\cdots=\alpha^2\beta\gamma$ determines $T_5, T_6$. For $\alpha_2\gamma_3\cdots=\alpha^2\beta\gamma$, the same argument repeats and we get the tiling in Figure \ref{Fig-a3-abab-Tiling-al2bega-al5} from the equatorial (first picture) and polar view (second picture). 


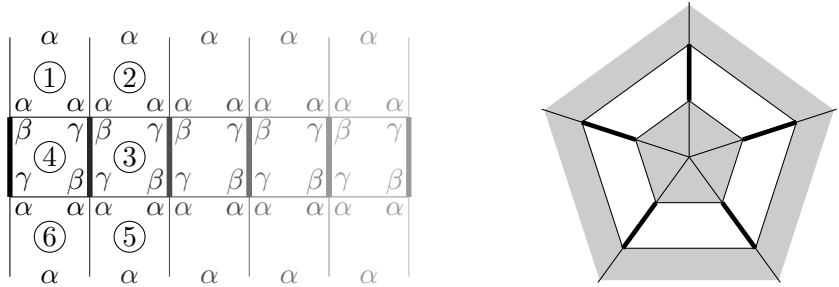
\begin{figure}[h!]
\centering
\begin{tikzpicture}[]

\tikzmath{
\s=1;
\r=0.75;
\th=360/4;
\x = \r*cos(0.5*\th);
\ph=360/5;
\sh=\r;
\ss=2;
\e=1.35;
}

\begin{scope}[] 

\tikzmath{
\tz=5;
\tzz=\tz-1;
}

\foreach \aa in {0,...,\tzz} {

\tikzmath{\c=100-\aa*15;}

\tikzset{shift={(\aa*2*\x,0)}}

\foreach \a in {0,...,3} {

\draw[black!\c, rotate=\a*\th]
	(0.5*\th:\r) -- (1.5*\th:\r)
;

}

\draw[black!\c, line width=2]
	(0.5*\th:\r) -- (-0.5*\th:\r)
	(1.5*\th:\r) -- (2.5*\th:\r)
;

}

\foreach \aa in {0,...,\tz} {

\tikzmath{\c=100-\aa*15;}

\draw[black!\c, shift={(\aa*2*\x,0)}]
	(-\x, \x) -- (-\x, 3*\x)
	(-\x, -\x) -- (-\x, -3*\x)
;

}

\foreach \aa in {0,...,\tzz} {

\tikzmath{\c=100-\aa*15;}

\foreach \a in {-1,1} {

\tikzset{yscale=\a}

\node at (0+\aa*2*\x,3*\x) {\small \textcolor{black!\c}{$\alpha$}};
\node at (-0.65*\x+\aa*2*\x,1.3*\x) {\small \textcolor{black!\c}{$\alpha$}};
\node at (0.65*\x+\aa*2*\x,1.3*\x) {\small \textcolor{black!\c}{$\alpha$}};

}

\foreach \a in {0,1} {

\tikzset{shift={(0+\aa*2*\x,0)}, rotate=\a*180}

\node at (-0.65*\x,0.625*\x) {\small \textcolor{black!\c}{$\beta$}};
\node at (0.65*\x,0.625*\x) {\small \textcolor{black!\c}{$\gamma$}};

}

}

\node[inner sep=1,draw,shape=circle] at (0,2*\x) {\small $1$};
\node[inner sep=1,draw,shape=circle] at (2*\x,2*\x) {\small $2$};
\node[inner sep=1,draw,shape=circle] at (2*\x,0) {\small $3$};
\node[inner sep=1,draw,shape=circle] at (0,0) {\small $4$};

\node[inner sep=1,draw,shape=circle] at (2*\x,-2*\x) {\small $5$};
\node[inner sep=1,draw,shape=circle] at (0,-2*\x) {\small $6$};

\end{scope} 

\begin{scope}[xshift=8.5*\s cm] 

\fill[gray!40]
	(90:\e*2*\r) -- (90-\ph:\e*2*\r) -- (90-2*\ph:\e*2*\r) -- (90-3*\ph:\e*2*\r) -- (90-4*\ph:\e*2*\r) --  cycle
;

\fill[white]
	(90:2*\r) -- (90-\ph:2*\r) -- (90-2*\ph:2*\r) -- (90-3*\ph:2*\r) -- (90-4*\ph:2*\r) --  cycle
;

\fill[gray!40]
	(90:\r) -- (90-\ph:\r) -- (90-2*\ph:\r) -- (90-3*\ph:\r) -- (90-4*\ph:\r) --  cycle
;

\foreach \a in {0,...,4} {

\draw[rotate=\ph*\a]
	(0,0) -- (90-1*\ph:\r)
	(90:\r) -- (90-1*\ph:\r)
	(90:\r) -- (90:2*\r)
	(90:2*\r) -- (90-1*\ph: 2*\r)
	(90:2*\r) -- (90:2.75*\r)
;

\draw[rotate=\ph*\a, line width=1.75]
	(90-1*\ph:\r) -- (90-1*\ph:2*\r) 
;

}

\end{scope}

\end{tikzpicture}
\caption{The tiling with $\alpha^2\beta\gamma, \alpha^5$}
\label{Fig-a3-abab-Tiling-al2bega-al5}
\end{figure}

The tiling is given by a pentagonal prism with triangulated pentagons. It can also be viewed as a tiling of earth map type in the first picture of Figure \ref{Fig-a3-abab-Tiling-al2bega-al5} and a timezone is given by a column of three tiles consisting of one triangle on top, another one at the bottom and one quadrilateral in the middle.

The angle sums of $\alpha^2\beta\gamma, \alpha^5$ and $\cos x = \cos^2 x + \sin^2 x \cos \alpha$ imply
\[
\alpha = \tfrac{2}{5}\pi, \quad
\beta + \gamma = \tfrac{6}{5}\pi, \quad
x = \cos^{-1} \tfrac{1}{\sqrt{5}}. 
\]
For fixed choices of $\beta, \gamma$, the governing identities for the quadrilaterals can be seen in \cite{ch}, whereby one can determine $y$.
\end{proof}

\begin{prop}\label{Prop-a3-abab-albe2} The dihedral tilings with $\alpha\beta^2$ are the first, the third and the fourth tiling in Figure \ref{Fig-a3-abab-Sporadic-Tilings}.
\end{prop}

\begin{proof} By \eqref{Eq-a3-abab-be2}, we have $\beta^2\cdots=\alpha\beta^2$. By $\beta>\gamma$ and $\alpha\beta^2$, we know that $\alpha\beta\gamma$ is not a vertex. By $\beta+\gamma>\pi$ and $\alpha\beta^2$, we get $2\gamma>\alpha$. 

Assume $\alpha\ge \beta$. By $\alpha\beta^2$ and \eqref{Eq-a3-abab-bega}, we know that $\alpha^2\beta\gamma, \alpha^2\beta\gamma^3, \alpha^2\beta\gamma^3$ are not vertices. By $\alpha+\gamma\ge\beta+\gamma>\pi$, we know that $\alpha\beta\gamma^3$ is not a vertex. Combined with no $\alpha\beta\gamma$, none of \eqref{Eq-a3-abab-bega} is a vertex. Hence $\beta\gamma\cdots$ is not a vertex. 

By $\alpha\beta^2$ and Counting Lemma, we know that $\gamma^2\cdots$ is a vertex. The same argument in Proposition \ref{Prop-am-abab-albe2} shows $\beta>\frac{2}{3}\pi>\alpha$, contradicting the assumption $\alpha\ge \beta$.


Hence we have $\alpha < \beta$. Recall $\alpha^a=\alpha^3, \alpha^4, \alpha^5$. By $\beta>\alpha$ and $\alpha\beta^2$, we get $\alpha^a=\alpha^4, \alpha^5$. By $2\gamma>\alpha>\frac{1}{3}\pi$ and $\beta+\gamma>\pi$, we have $2\alpha+\beta+3\gamma>2\pi$. Then by $\alpha\beta^2$ and \eqref{Eq-a3-abab-bega} and no $\alpha\beta\gamma$, we get $\beta\gamma\cdots=\alpha^2\beta\gamma, \alpha\beta\gamma^3$. Combined with $\beta^2\cdots=\alpha\beta^2$, we conclude $\beta\cdots = \alpha\beta^2, \alpha^2\beta\gamma, \alpha\beta\gamma^3$. Then by \eqref{Eq-a3/am-abab-ga2} and Parity Lemma, we conclude $\gamma\cdots=\alpha^2\beta\gamma, \alpha\beta\gamma^3, \alpha^{a\ge2}\gamma^{c\ge2}$ where $c$ is even. Hence we get the list of vertices below,
\begin{align}\label{Eq-list-albe2}
\alpha\beta^2, \alpha^2\beta\gamma, \alpha\beta\gamma^3, \alpha^4, \alpha^5, \alpha^{a\ge2}\gamma^c.
\end{align}
From the above, we have $\beta^2\cdots=\alpha\beta^2 = \vert \alpha \vert \beta \, \bvert \, \beta \vert$. Then $\beta \vert \beta \cdots$ is not a vertex. By the second picture of Figure \ref{Fig-a4-abab-adj-square-quad}, it implies that $\gamma \vert \gamma \cdots$ is also not a vertex. Moreover, recall the discussion after \eqref{Eq-a3/am-abab-ga2}, we have $2a \ge c \ge 2$ in $\alpha^{a\ge2}\gamma^c$.

We divide the discussion into the cases of $2\alpha+\beta+\gamma=2\pi$ and $2\alpha+\beta+\gamma\neq2\pi$.

\begin{case*}[$2\alpha+\beta+\gamma=2\pi$] By $\beta>\alpha$ and $\alpha\beta^2$, we have $\pi>\beta>\frac{2}{3}\pi>\alpha$. By $2\alpha+\beta+\gamma=2\pi$ and $\beta+\gamma>\pi$, we get $\alpha<\frac{1}{2}\pi$. By $2\alpha+\beta+\gamma=2\pi$ and $\alpha\beta^2$, we get $\alpha+\gamma=\beta>\frac{2}{3}\pi$ and $3\alpha+2\gamma=2\pi$.

From \eqref{Eq-list-albe2}, we know that $\gamma^2 \cdots=\alpha\beta\gamma^3, \alpha^{a\ge2}\gamma^c$ is a vertex. 

The assumption $2\alpha+\beta+\gamma=2\pi$ and $\alpha\beta^2, \alpha\beta\gamma^3$ imply $\alpha=\frac{1}{2}\pi$, contradicting $\alpha<\frac{1}{2}\pi$. So $\alpha\beta\gamma^3$ is not a vertex.

If $c\ge4$ in $\alpha^{a\ge2}\gamma^c$, then by $2a\ge c$ we have $\alpha^2\gamma^4\cdots$. By $2\gamma>\alpha$, we get $2\alpha+4\gamma>3\alpha+2\gamma=2\pi$, which means that $\alpha^2\gamma^4\cdots$ is not a vertex. 

If $c=2$ in $\alpha^{a\ge2}\gamma^c$, then by $3\alpha+2\gamma=2\pi$, we get $\alpha^a\gamma^2 = \alpha^3\gamma^2=$ $\bvert \, \gamma \vert \alpha \vert \alpha \vert \alpha \vert \gamma \, \bvert$. 

Hence we have $\gamma^2\cdots=\alpha^3\gamma^2$. By $2\gamma>\alpha$ and $\alpha^3\gamma^2$, we know that $\alpha^4$ is not a vertex. If $\alpha^5$ is also a vertex, then we get $\alpha=\gamma$ and then $\beta=2\gamma$. So the quadrilateral can be subdivided into two equilateral triangles and $x=y$, a contradiction. So $\alpha^5$ is not a vertex. 

We therefore update the vertices below,
\begin{align} \label{Eq-am-abab-AVC-albe2-al2bega-al3ga2}
\AVC = \{ \alpha\beta^2, \alpha^2\beta\gamma,  \alpha^3\gamma^2 \}.
\end{align}
From the above, we know $\beta^2 \cdots = \alpha\beta^2$ and $\gamma^2 \cdots= \alpha^3\gamma^2$ and $\beta\gamma \cdots=\alpha^2\beta\gamma$. 

The vertex $\alpha^3\gamma^2$ determines tiles $T_1, T_2, T_3, T_4, T_5$ in Figure \ref{Fig-a3-abab-Tiling-albe2-al3ga2}. Then $\beta_4\beta_5 \cdots=\alpha\beta^2$ determines $T_6$. Adjacent to $T_4, T_5$, we have $\varphi, \theta = \beta$ or $\gamma$. 

\begin{figure}[h!] 
\centering
\begin{tikzpicture}

\tikzmath{
\s=1;
\r=1.5;
\ph=360/6;
\x=\r*cos(0.5*\ph);
\y=\r*sin(0.5*\ph);
\R=2*\r;
\X=\x;
}

\begin{scope}[]

\foreach \a in {1,...,5} {

\draw[rotate=\a*\ph]
	(0:0) -- (90:\r) 
;

}

\foreach \a in {0,3} {

\draw[rotate=\a*\ph]
	(90-2*\ph:\r) -- (90-1*\ph:\r)
;

}

\foreach \aa in {-1, 1} {

\tikzset{xscale=\aa}

\draw[]
	(90-\ph:\r) -- (90+\ph:\r)
	(270:\r) -- (\x,-3*\y)	
	(-\x,-3*\y) -- (\x,-3*\y)
;	

\draw[line width=2]
	(0:0) -- (270:\r)
	(\x,-\y) -- (\x,-3*\y)	
;

\node at (0.25*\x, 0) {\small $\alpha$};
\node at (0.85*\x, 0.35*\x) {\small $\alpha$};
\node at (0.85*\x, -0.35*\x) {\small $\alpha$};

\node at (0.2*\x, -0.3*\x) {\small $\gamma$};
\node at (0.2*\x, -1.1*\x) {\small $\beta$};
\node at (0.8*\x, -0.75*\x) {\small $\beta$};
\node at (0.8*\x, -1.4*\x) {\small $\gamma$};

}

\node at (0, 0.15*\x) {\small $\alpha$};
\node at (0.6*\x, 0.5*\x) {\small $\alpha$};
\node at (-0.6*\x, 0.5*\x) {\small $\alpha$};

\node at (0,-1.125*\r) {\small $\alpha$};
\node at (0.65*\x,-1.65*\x) {\small $\alpha$};
\node at (-0.65*\x,-1.65*\x) {\small $\alpha$};

\node at (-1.2*\x, -0.75*\x) {\small $\theta$};
\node at (1.2*\x, -0.75*\x) {\small $\varphi$};

\node[inner sep=1,draw,shape=circle] at (-0.7*\x,0) {\small $1$};
\node[inner sep=1,draw,shape=circle] at (0,0.4*\x) {\small $2$};
\node[inner sep=1,draw,shape=circle] at (0.7*\x,0) {\small $3$};
\node[inner sep=1,draw,shape=circle] at (-0.45*\x,-0.8*\x) {\small $5$};
\node[inner sep=1,draw,shape=circle] at (0.45*\x,-0.8*\x) {\small $4$};

\node[inner sep=1,draw,shape=circle] at (0,-1.35*\r) {\small $6$};

\end{scope}

\begin{scope}[xshift=6*\s cm] 

\foreach \a in {1,...,5} {

\draw[rotate=\a*\ph]
	(0:0) -- (90:\r) 
;

}

\foreach \a in {0,1,...,3} {

\draw[rotate=\a*\ph]
	(90-2*\ph:\r) -- (90-1*\ph:\r)
;

}

\foreach \aa in {-1, 1} {

\tikzset{xscale=\aa}

\draw[]
	(270:\r) -- (\x,-3*\y)	
	(\x,-\y) -- (\x,-3*\y)	
	(\x,-3*\y) -- (1.6*\x, 0.5*\y)
	(\x,\y) -- (1.6*\x, 0.5*\y)
	(90:\r) -- (90:2*\x)
	(90:2*\x) -- (1.6*\x, 0.5*\y)
;	

\draw[line width=2]
	(\x,-\y) -- (\x,-3*\y)	
	(\x,\y) -- (1.6*\x, 0.5*\y)
;

\node at (0.25*\x, 0) {\small $\alpha$};
\node at (0.85*\x, 0.35*\x) {\small $\alpha$};
\node at (0.85*\x, -0.35*\x) {\small $\alpha$};

\node at (0.2*\x, -0.3*\x) {\small $\gamma$};
\node at (0.2*\x, -1.1*\x) {\small $\beta$};
\node at (0.8*\x, -0.75*\x) {\small $\beta$};
\node at (0.8*\x, -1.4*\x) {\small $\gamma$};

\node at (1.15*\x, -0.45*\x) {\small $\beta$};
\node at (1.1*\x, -1*\x) {\small $\gamma$};
\node at (1.1*\x, 0.35*\x) {\small $\gamma$};
\node at (1.45*\x, 0.15*\x) {\small $\beta$};

\node at (0.15*\x,1.25*\x) {\small $\beta$};
\node at (1*\x, 0.7*\x) {\small $\gamma$};
\node at (0.15*\x,1.65*\x) {\small $\gamma$};

\node at (1.75*\x, 0.25*\x) {\small $\alpha$};
\node at (0.35*\x,1.85*\x) {\small $\alpha$};
\node at (1.25*\x,-1.4*\x) {\small $\alpha$};

}

\draw[]
	(-\x,0.5*\r) -- (\x,0.5*\r)
;

\draw[] (0,0) circle (2*\x); 

\draw[line width=2]
	(0:0) -- (270:\r)
	(90:\r) -- (90:2*\x)
;

\node at (0, 0.15*\x) {\small $\alpha$};
\node at (0.6*\x, 0.675*\x) {\small $\alpha$};
\node at (-0.6*\x, 0.675*\x) {\small $\alpha$};

\node at (0, 1*\x) {\small $\alpha$};
\node at (0.6*\x, 0.475*\x) {\small $\alpha$};
\node at (-0.6*\x, 0.475*\x) {\small $\alpha$};

\node at (0,-1.15*\r) {\small $\alpha$};
\node at (0.75*\x,-1.725*\x) {\small $\alpha$};
\node at (-0.75*\x,-1.725*\x) {\small $\alpha$};

\node at (0, \r+\x) {\small $\alpha$};
\node at (1.15*\x, -3.15*\y) {\small $\alpha$};
\node at (-1.15*\x, -3.15*\y) {\small $\alpha$};

\node[inner sep=1,draw,shape=circle] at (-0.7*\x,0) {\small $1$};
\node[inner sep=1,draw,shape=circle] at (0,0.4*\x) {\small $2$};
\node[inner sep=1,draw,shape=circle] at (0.7*\x,0) {\small $3$};
\node[inner sep=1,draw,shape=circle] at (-0.45*\x,-0.8*\x) {\small $5$};
\node[inner sep=1,draw,shape=circle] at (0.45*\x,-0.8*\x) {\small $4$};

\node[inner sep=1,draw,shape=circle] at (0,-1.5*\r) {\small $6$};
\node[inner sep=1,draw,shape=circle] at (-1.25*\x,0*\x) {\small $7$};
\node[inner sep=1,draw,shape=circle] at (1.25*\x,0*\x) {\small $8$};
\node[inner sep=1,draw,shape=circle] at (0,0.75*\x) {\small $9$};
\node[inner sep=0.25,draw,shape=circle] at (-0.45*\x,1.2*\x) {\footnotesize $10$};
\node[inner sep=0.25,draw,shape=circle] at (0.45*\x,1.2*\x) {\footnotesize $11$};
\node[inner sep=0.25,draw,shape=circle] at (-1.2*\x,1.2*\x) {\footnotesize $12$};

\node[inner sep=0.25,draw,shape=circle] at (1.2*\x,1.2*\x) {\footnotesize $13$};

\node[inner sep=0.25,draw,shape=circle] at (2.25*\r,0) {\footnotesize $14$};

\end{scope}

\end{tikzpicture}
\caption{The tiling with $\alpha\beta^2, \alpha^3\gamma^2$}
\label{Fig-a3-abab-Tiling-albe2-al3ga2}
\end{figure}
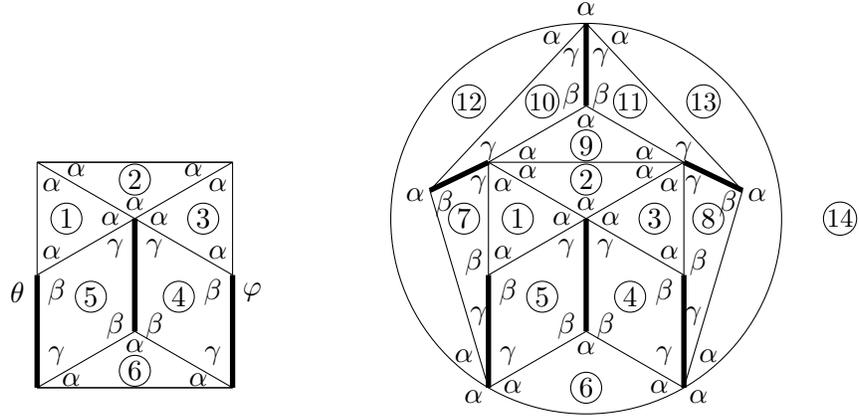

If both $\varphi, \theta=\beta$, then we have tiles $T_1, ..., T_8$ in the second picture of Figure \ref{Fig-a3-abab-Tiling-albe2-al3ga2}. If the adjacent tile $T_9$ above $T_2$ is quadrilateral, then one of $\alpha_1\alpha_2\gamma_7\cdots, \alpha_2\alpha_3\gamma_8\cdots$ is $\alpha^3\gamma^2$ with angle arrangement $\bvert  \, \gamma \vert \alpha \vert \alpha \vert \gamma \, \bvert \cdots $, contradicting $\alpha^3\gamma^2=$ $\bvert \, \gamma \vert \alpha \vert \alpha \vert \alpha \vert \gamma \, \bvert$. Hence $T_9$ is a triangle as shown. By $\alpha_1\alpha_2\alpha_9\gamma_7\cdots$, $\alpha_2\alpha_3\alpha_9\gamma_8\cdots = \alpha^3\gamma^2$, we determine $T_{10}, T_{11}$. Then $\gamma_{10}\gamma_{11}\cdots=\alpha^3\gamma^2$ determines $T_{12}, T_{13}, T_{14}$. The tiling obtained is the fourth tiling in Figure \ref{Fig-a3-abab-Sporadic-Tilings} and $\alpha^2\beta\gamma$ does not appear as a vertex.

Up to mirror symmetry, it remains to consider $\theta = \gamma$. Then $\beta_2 \gamma_7\cdots=\alpha^2\beta\gamma$ determines $T_7, T_8$ in Figure \ref{Fig-a3-abab-Tiling-albe2-al2bega-al3ga2}. Similarly, $\beta_7\gamma_5\cdots=\alpha^2\beta\gamma$ determines $T_9$. By \eqref{Eq-am-abab-AVC-albe2-al2bega-al3ga2}, we have $\alpha_1\alpha_2\alpha_8\cdots=\alpha^3\gamma^2$. This further determines $T_{10}, T_{11}$. Then $\alpha_2\alpha_3\beta_{11}\cdots=\alpha^2\beta\gamma$ determines $T_{12}$. By $\beta_{12}\gamma_{11}\cdots=\alpha^2\beta\gamma$, we determine $T_{13}, T_{14}$. The tiling obtained is the third tiling in Figure \ref{Fig-a3-abab-Sporadic-Tilings}.

\begin{figure}[h!] 
\centering
\begin{tikzpicture}

\tikzmath{
\s=1;
\r=1.5;
\ph=360/6;
\x=\r*cos(0.5*\ph);
\y=\r*sin(0.5*\ph);
\R=2*\r;
\X=\x;
}

\foreach \a in {1,...,5} {

\draw[rotate=\a*\ph]
	(0:0) -- (90:\r) 
;

}

\foreach \a in {0,3} {

\draw[rotate=\a*\ph]
	(90-2*\ph:\r) -- (90-1*\ph:\r)
;

}

\foreach \aa in {-1, 1} {

\tikzset{xscale=\aa}

\draw[]
	(90-\ph:\r) -- (90+\ph:\r)
	(270:\r) -- (\x,-3*\y)	
	(-\x,-3*\y) -- (\x,-3*\y)
;	

\draw[line width=2]
	(0:0) -- (270:\r)
	(\x,-\y) -- (\x,-3*\y)	
;

\node at (0.25*\x, 0) {\small $\alpha$};
\node at (0.85*\x, 0.35*\x) {\small $\alpha$};
\node at (0.85*\x, -0.35*\x) {\small $\alpha$};

\node at (0.2*\x, -0.3*\x) {\small $\gamma$};
\node at (0.2*\x, -1.1*\x) {\small $\beta$};
\node at (0.8*\x, -0.75*\x) {\small $\beta$};
\node at (0.8*\x, -1.4*\x) {\small $\gamma$};

}

\draw[]
	(-2*\x,0) -- (-\x, \y)
	(-2*\x,0) -- (-\x, -\y)
	(\x,-3*\y) -- (0,-\R-0.5*\r)
	(-\x,-3*\y) -- (0,-\R-0.5*\r)
	%
	(\x,-3*\y) -- ([shift={(0,-0.5*\r)}]90-\ph:\R)
;

\arcThroughThreePoints[line width=2]{-2*\X,0}{[shift={(270:\r)}]90+2.5*\ph:1.3*\r}{270:\R+0.5*\r};

\draw[line width=2]
	(90-\ph:\r) -- ([shift={(0,-0.5*\r)}]90-\ph:\R)
	(90+\ph:\r) -- ([shift={(0,-0.5*\r)}]90+\ph:\R)
;

\draw[] (0,-0.5*\r) circle (\R);

\node at (0, 0.15*\x) {\small $\alpha$};
\node at (0.6*\x, 0.5*\x) {\small $\alpha$};
\node at (-0.6*\x, 0.5*\x) {\small $\alpha$};

\node at (0,-1.125*\r) {\small $\alpha$};
\node at (0.65*\x,-1.65*\x) {\small $\alpha$};
\node at (-0.65*\x,-1.65*\x) {\small $\alpha$};

\node at (-1.15*\x, -0.75*\x) {\small $\gamma$};
\node at (-1.85*\x, -0.275*\x) {\small $\beta$};
\node at (-1.15*\x, -3*\y) {\small $\beta$};
\node at (-0.5*\x,-2.525*\x) {\small $\gamma$};

\node at (-1.75*\x, 0) {\small $\alpha$};
\node at (-1.15*\x, 0.35*\x) {\small $\alpha$};
\node at (-1.15*\x, -0.35*\x) {\small $\alpha$};

\node at (-1.5*\x, 0.4*\x) {\small $\gamma$};
\node at (-1.95*\x, 0.4*\x) {\small $\beta$};
\node at (-2.1*\x, 0) {\small $\beta$};
\node at (-1*\x,-2.5*\x) {\small $\gamma$};

\node at (\x, 0.75*\x) {\small $\beta$};
\node at (-\x, 0.75*\x) {\small $\gamma$};
\node at (1.725*\x, 0.75*\x) {\small $\gamma$};
\node at (-1.725*\x, 0.75*\x) {\small $\beta$};

\node at (1.15*\x, -0.5*\x) {\small $\beta$};
\node at (1.15*\x, 0.4*\x) {\small $\gamma$};
\node at (1.15*\x, -1.15*\x) {\small $\gamma$};
\node at (1.75*\x, 0.35*\x) {\small $\beta$};

\node at (1.15*\x,-1.75*\x) {\small $\alpha$};
\node at (0.25*\x,-\R-0.4*\r) {\small $\alpha$};
\node at (2*\x, 0.25*\x) {\small $\alpha$};

\node at (-0.75*\x,-1.85*\x) {\small $\alpha$};
\node at (0.75*\x,-1.85*\x) {\small $\alpha$};
\node at (0,-\R-0.3*\r) {\small $\alpha$};

\node at (0,-\R-0.6*\r) {\small $\alpha$};
\node at (2.1*\x, 0.65*\x) {\small $\alpha$};
\node at (-2.1*\x, 0.65*\x) {\small $\alpha$};

\node[inner sep=1,draw,shape=circle] at (-0.7*\x,0) {\small $1$};
\node[inner sep=1,draw,shape=circle] at (0,0.4*\x) {\small $2$};
\node[inner sep=1,draw,shape=circle] at (0.7*\x,0) {\small $3$};
\node[inner sep=1,draw,shape=circle] at (-0.45*\x,-0.8*\x) {\small $5$};
\node[inner sep=1,draw,shape=circle] at (0.45*\x,-0.8*\x) {\small $4$};

\node[inner sep=1,draw,shape=circle] at (0,-1.325*\r) {\small $6$};
\node[inner sep=1,draw,shape=circle] at (-1.45*\x,-1*\r) {\small $7$};
\node[inner sep=1,draw,shape=circle] at (-1.3*\x,0) {\small $8$};
\node[inner sep=1,draw,shape=circle] at (0,-1.8*\r) {\small  $9$};
\node[inner sep=0.25,draw,shape=circle] at (-2*\x,-1*\r) {\footnotesize $10$};
\node[inner sep=0.25,draw,shape=circle] at (0,1*\x) {\footnotesize $11$};
\node[inner sep=0.25,draw,shape=circle] at (1.35*\x,0) {\footnotesize $12$};
\node[inner sep=0.25,draw,shape=circle] at (1.5*\x,-1.5*\x) {\footnotesize $13$};
\node[inner sep=0.25,draw,shape=circle] at (2.75*\x,-0.5*\x) {\footnotesize $14$};

\end{tikzpicture}
\caption{The tiling with $\alpha\beta^2, \alpha^2\beta\gamma, \alpha^3\gamma^2$}
\label{Fig-a3-abab-Tiling-albe2-al2bega-al3ga2}
\end{figure}
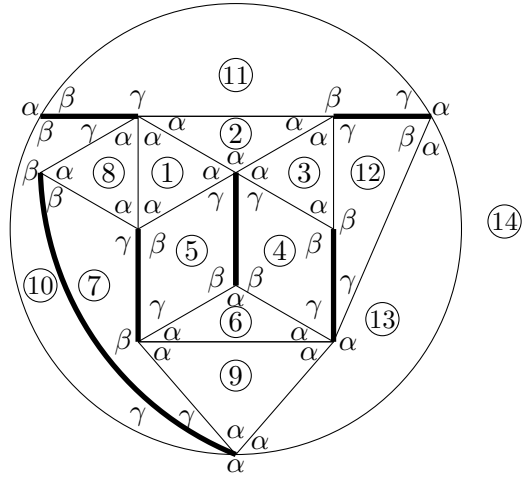

The tilings in Figures \ref{Fig-a3-abab-Tiling-albe2-al3ga2}, \ref{Fig-a3-abab-Tiling-albe2-al2bega-al3ga2} are given by the hexagonal prism with triangulated hexagons. By $\AVC$ \eqref{Eq-am-abab-AVC-albe2-al2bega-al3ga2}, we get
\begin{align} \label{Eq-am-abab-angles-albe2-al2bega-al3ga2}
\beta= \pi - \tfrac{1}{2}\alpha, \quad
\gamma = \pi - \tfrac{3}{2}\alpha.
\end{align}

When $x=y$, from \cite[Equation (2.10)]{cl3} we have
\begin{align*} 
\cot^2 \tfrac{1}{2}\alpha + \tfrac{\cos \frac{2}{3}\pi}{ \sin^2 \frac{1}{2}\alpha } = \cot \tfrac{1}{2}\beta \cot \tfrac{1}{2}\gamma.
\end{align*}
Substitute \eqref{Eq-am-abab-angles-albe2-al2bega-al3ga2} into the above, we get
\begin{align*}
5t^4 - 10t^2 + 1 = 0, \quad \text{where } t:=\tan \tfrac{1}{4}\alpha.
\end{align*}
The roots are
\begin{align*}
t = \pm \sqrt{ \sqrt{5} - 2 }, \  \pm  \sqrt{ \sqrt{5} + 2 }.
\end{align*}
For $t=\tan \frac{1}{4}\alpha$ and $\beta>\alpha>0$ and $\gamma>0$ and $\cos x = \cos^2 x + \sin^2 x \cos \alpha$ (cosine law on the equilateral triangle with $\alpha$), we further get
\begin{align*}
\alpha = \tfrac{2}{5}\pi, \quad
\beta  = \tfrac{4}{5}\pi, \quad
\gamma = \tfrac{2}{5}\pi, \quad
x = \cos^{-1} \tfrac{1}{\sqrt{5}}.
\end{align*}
Then for $x\neq y$, the tiling is the perturbation of the equality case. 
\end{case*}

\begin{case*}[$2\alpha+\beta+\gamma \neq 2\pi$] By $2\alpha+\beta+\gamma \neq 2\pi$, we know that $\alpha^2\beta\gamma$ is not a vertex. Then by \eqref{Eq-list-albe2}, we have $\beta\cdots=\alpha\beta^2, \alpha\beta\gamma^3$. By no $\gamma\vert\gamma\cdots$, the latter has a unique angle arrangement $\bvert \, \beta \vert \gamma \, \bvert \, \gamma \vert \alpha \vert \gamma \, \bvert$. This further implies $\alpha\vert\beta\cdots=\alpha\beta^2$. 

The angle arrangement $\bvert \, \gamma \vert \alpha \vert \gamma \, \bvert$ determines $T_1, T_2, T_3$ in Figure \ref{Fig-a3-abab-albe2-AAD-gaalga}. By $\alpha_1 \vert \beta_2 \cdots, \alpha_1\vert\beta_3 \cdots=\alpha\beta^2$, we get two consecutive $\beta$'s in $T_4$, a contradiction. So $\gamma \vert \alpha \vert \gamma \cdots$ is not a vertex. This implies $\beta\cdots=\alpha\beta^2$. 

\begin{figure}[h!] 
\centering
\begin{tikzpicture}

\tikzmath{
\r=0.8;
\th=360/3;
}

\begin{scope}[]

\foreach \t in {0,...,2} {

\draw[rotate=\t*\th]
	(90:\r) -- (90+\th:\r)
;

\node at (90+\t*\th:0.6*\r) {\small $\alpha$};

}

\foreach \t in {0,2} {

\draw[line width=1.75, rotate=\t*\th]
	(90-\th:\r) -- (90-\th:1.5*\r) 
;

}

\draw[line width=1.75]
	(90:\r) -- ([shift={(90:\r)}]45:0.5*\r) 
	(90:\r) -- ([shift={(90:\r)}]135:0.5*\r) 
;

\node at (90-0.15*\th:1*\r) {\small $\gamma$};
\node at (90+0.15*\th:1*\r) {\small $\gamma$};

\node at (90-0.85*\th:1*\r) {\small $\beta$};
\node at (90-1.15*\th:1*\r) {\small $\beta$};

\node at (90+1.15*\th:1*\r) {\small $\beta$};
\node at (90+0.85*\th:1*\r) {\small $\beta$};

\node at (88:1.5*\r) {\small $\cdots$};

\node[inner sep=1,draw,shape=circle] at (0,0) {\small $1$};
\node[inner sep=1,draw,shape=circle] at (90-0.5*\th:\r) {\small $2$};
\node[inner sep=1,draw,shape=circle] at (90+0.5*\th:\r) {\small $3$};
\node[inner sep=1,draw,shape=circle] at (270:\r) {\small $4$};

\end{scope}

\end{tikzpicture}
\caption{The angle arrangement of $\bvert \, \gamma \vert \alpha \vert \gamma \, \bvert$}
\label{Fig-a3-abab-albe2-AAD-gaalga}
\end{figure}
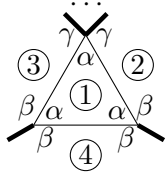

By $\beta\cdots=\alpha\beta^2$ and Parity Lemma, we know $\gamma\cdots=\gamma^2\cdots=\alpha^a\gamma^c$ where $a\ge c\ge2$ and $c$ is even. By $\alpha>\frac{1}{3}\pi$ and $\gamma>\frac{1}{6}\pi$, we get $4\alpha+4\gamma, 5\alpha+2\gamma > 2\pi$. Then by $a\ge c\ge2$ and $c$ being even, we get $\alpha^a\gamma^2=\alpha^2\gamma^2, \alpha^3\gamma^2, \alpha^4\gamma^2$, which are mutually exclusive. When $\alpha^4\gamma^2$ is a vertex, then $2\gamma > \alpha$ implies that $\alpha^4, \alpha^5$ are not vertices. Similarly, $\alpha^3\gamma^2$ implies no $\alpha^4$. Hence \eqref{Eq-list-albe2} gives
\begin{align*}
\AVC &= \{ \alpha\beta^2, \alpha^2\gamma^2, \alpha^4, \alpha^5 \}; \\
\AVC &= \{ \alpha\beta^2, \alpha^3\gamma^2, \alpha^5 \}; \\
\AVC &= \{ \alpha\beta^2, \alpha^4\gamma^2 \}.
\end{align*}
From the above, we have $\beta\cdots=\alpha\beta^2$ and $\gamma\cdots=\alpha^2\gamma^2, \alpha^3\gamma^2, \alpha^4\gamma^2$. 

For $\AVC = \{ \alpha\beta^2, \alpha^4\gamma^2 \}$, the vertex $\alpha^4\gamma^2$ determines tiles $T_1, T_2,..., T_6$ in Figure \ref{Fig-a3-abab-AAD-albe2-al4ga2}. Then $\alpha_1\beta_6\cdots=\alpha\beta^2$ determines $T_7$. By $\alpha_1\gamma_7\cdots=\alpha^4\gamma^2$, we further determine $T_8, T_9, T_{10}$. By mirror symmetry, we also get $T_{11}, ..., T_{15}$. Then $\alpha_2\alpha_3\alpha_{10}\alpha_{16}\cdots=\alpha^4\gamma^2$ determines $T_{16}, T_{17}$. It implies $\alpha^2\beta\cdots$, a contradiction.

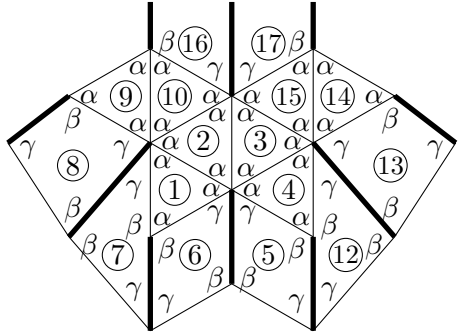
\begin{figure}[h!] 
\centering
\begin{tikzpicture}

\tikzmath{
\s=1;
\r=1.25;
\ph=360/6;
\x=\r*cos(0.5*\ph);
\y=\r*sin(0.5*\ph);
}

\foreach \a in {0,1,...,5} {

\draw[rotate=\a*\ph]
	(0:0) -- (90:\r) 
;

}

\foreach \a in {0,1,...,3} {

\draw[rotate=\a*\ph]
	(90-2*\ph:\r) -- (90-1*\ph:\r)
;

}

\foreach \aa in {-1, 1} {

\tikzset{xscale=\aa}

\draw[]
	(270:\r) -- (\x,-3*\y)	
	(\x,-\y) -- (\x,-3*\y)	
	(90:\r) -- (\x,3*\y)	
	(\x, \y) -- (\x,3*\y)	
	(\x, \y) -- (2*\x, -\y)
	(\x,-3*\y) -- (2*\x, -\y)
	(\x, \y) -- (2*\x, 2*\y)
	(\x, 3*\y) -- (2*\x, 2*\y)
	(2*\x, 2*\y) -- (2.75*\x, \y)
	(2*\x, -\y) -- (2.75*\x, \y)
	%
;	

\draw[line width=2]
	(\x,-\y) -- (\x,-3*\y)	
	(\x, \y) -- (2*\x, -\y)
	(2*\x, 2*\y) -- (2.75*\x, \y)
	(\x,3*\y) --  (\x,4*\y)	 
;

\node at (0.25*\x, 0) {\small $\alpha$};
\node at (0.85*\x, 0.35*\x) {\small $\alpha$};
\node at (0.85*\x, -0.35*\x) {\small $\alpha$};

\node at (0.175*\x, 0.25*\x) {\small $\alpha$};
\node at (0.175*\x, 0.925*\x) {\small $\alpha$};
\node at (0.75*\x, 0.575*\x) {\small $\alpha$};

\node at (0.2*\x, -0.3*\x) {\small $\gamma$};
\node at (0.2*\x, -1.1*\x) {\small $\beta$};
\node at (0.8*\x, -0.75*\x) {\small $\beta$};
\node at (0.8*\x, -1.4*\x) {\small $\gamma$};

\node at (1.2*\x, -0.45*\x) {\small $\beta$};
\node at (1.2*\x, 0) {\small $\gamma$};
\node at (1.75*\x, -0.7*\x) {\small $\beta$};
\node at (1.2*\x, -1.25*\x) {\small $\gamma$};

\node at (1.35*\x, 0.5*\x) {\small $\gamma$};
\node at (1.95*\x, 0.85*\x) {\small $\beta$};
\node at (1.95*\x, -0.25*\x) {\small $\beta$};
\node at (2.5*\x, 0.5*\x) {\small $\gamma$};

\node at (1.75*\x, 1.15*\x) {\small $\alpha$};
\node at (1.15*\x, 0.8*\x) {\small $\alpha$};
\node at (1.15*\x, 1.5*\x) {\small $\alpha$};

\node at (0.25*\x, 1.15*\x) {\small $\alpha$};
\node at (0.85*\x, 0.8*\x) {\small $\alpha$};
\node at (0.85*\x, 1.5*\x) {\small $\alpha$};



\node at (0.2*\x, 1.45*\x) {\small $\gamma$};
\node at (0.8*\x, 1.8*\x) {\small $\beta$};

}

\draw[line width=2]
	(0,0) -- (270:\r) 
	(90:\r) -- (90:2*\r)
;

\node[inner sep=1,draw,shape=circle] at (-0.7*\x,0) {\small $1$};
\node[inner sep=1,draw,shape=circle] at (-0.35*\x,0.6*\x) {\small $2$};
\node[inner sep=1,draw,shape=circle] at (0.35*\x,0.6*\x) {\small $3$};
\node[inner sep=1,draw,shape=circle] at (0.7*\x,0) {\small $4$};
\node[inner sep=1,draw,shape=circle] at (0.45*\x,-0.8*\x) {\small $5$};
\node[inner sep=1,draw,shape=circle] at (-0.45*\x,-0.8*\x) {\small $6$};

\node[inner sep=1,draw,shape=circle] at (-1.4*\x,-0.8*\x) {\small $7$};
\node[inner sep=1,draw,shape=circle] at (-1.95*\x,0.3*\x) {\small $8$};
\node[inner sep=1,draw,shape=circle] at (-1.3*\x,1.15*\x) {\small $9$};
\node[inner sep=0.25,draw,shape=circle] at (-0.7*\x,1.15*\x) {\footnotesize $10$};
\node[inner sep=0.25,draw,shape=circle] at (-0.45*\x,1.8*\x) {\footnotesize $16$};

\node[inner sep=0.25,draw,shape=circle] at (1.4*\x,-0.8*\x) {\footnotesize $12$};
\node[inner sep=0.25,draw,shape=circle] at (1.95*\x,0.3*\x) {\footnotesize $13$};
\node[inner sep=0.25,draw,shape=circle] at (1.3*\x,1.15*\x) {\footnotesize $14$};
\node[inner sep=0.25,draw,shape=circle] at (0.7*\x,1.15*\x) {\footnotesize $15$};
\node[inner sep=0.25,draw,shape=circle] at (0.45*\x,1.8*\x) {\footnotesize $17$};

\end{tikzpicture}
\caption{The angle arrangement of $\alpha^4\gamma^2$}
\label{Fig-a3-abab-AAD-albe2-al4ga2}
\end{figure}

For $\AVC = \{ \alpha\beta^2, \alpha^3\gamma^2, \alpha^5 \}$, the vertex $\alpha^3\gamma^2$ determines tiles $T_1, T_2, ..., T_6$ in the first picture of Figure \ref{Fig-a3-abab-Tiling-albe2-al3ga2}. By $\beta\cdots=\alpha\beta^2$, we have $\varphi, \theta=\beta$. The same argument in Figure \ref{Fig-a3-abab-Tiling-albe2-al3ga2} determines the same tiling in the second picture.

\begin{figure}[h!] 
\centering
\begin{tikzpicture}

\tikzmath{
\s=1;
}

\begin{scope}[]

\tikzmath{
\r=0.8;
\th=90;
\x=\r*cos(0.5*\th);
\X=\x;
\ps = asin ( 2*\x/(2.9*\x) );
}

\foreach \aa in {-1,1} {

\tikzset{shift={(\aa*\x,0)}}

\foreach \a in {0,...,3} {

\draw[rotate=\a*\th]
	(0.5*\th:\r) -- (1.5*\th:\r)
;

}

}

\foreach \aa in {-1,1} {

\draw[xscale=\aa]
	(0,3*\x) -- (2*\x, \x)
;

}

\draw[]
	(0,\x) -- (0,3*\x)
	(0,3*\x) -- (0,4*\x)
;

\draw[] (0,1.1*\x) circle (2.9*\x);

\foreach \aa in {-1,1} {

\tikzset{shift={(\aa*\x,0)}}

\foreach \a in {0,2} {

\tikzset{xscale=\aa}

\node at (0.5*\th+\a*\th:0.65*\r) {\small $\beta$};
\node at (1.5*\th+\a*\th:0.65*\r) {\small $\gamma$};

}

}

\foreach \aa in {-1,1} {

\tikzset{xscale=\aa}

\node at (0.3*\x,1.25*\x) {\small $\alpha$};
\node at (0.3*\x,2.35*\x) {\small $\alpha$};
\node at (1.35*\x,1.25*\x) {\small $\alpha$};

\node at (0.35*\x,3*\x) {\small $\gamma$};
\node at (0.35*\x,3.55*\x) {\small $\beta$};
\node at (2.25*\x,1*\x) {\small $\beta$};
\node at (2.25*\x,-0.45*\x) {\small $\gamma$};

}

\node at (0,-1.225*\x) {\small $\alpha$};
\node at (1.35*\x,-1.225*\x) {\small $\alpha$};
\node at (-1.35*\x,-1.225*\x) {\small $\alpha$};

\node at (0,4.25*\x) {\small $\alpha$};
\node at (2.35*\x,-\x) {\small $\alpha$};
\node at (-2.35*\x,-\x) {\small $\alpha$};

\draw[line width=2]
	(0,\x) --(0,-\x)
	(-2*\x,\x) --(-2*\x,-\x)
	(2*\x,\x) --(2*\x,-\x)
	(0,3*\x) -- (0,4*\x)
;

\node[inner sep=1,draw,shape=circle] at (-0.7*\x,1.7*\x) {\small $1$};
\node[inner sep=1,draw,shape=circle] at (0.7*\x,1.7*\x) {\small $2$};
\node[inner sep=1,draw,shape=circle] at (-1*\x,0) {\small $4$};
\node[inner sep=1,draw,shape=circle] at (1*\x,0) {\small $3$};

\node[inner sep=1,draw,shape=circle] at (-1.5*\x,2.5*\x) {\small $5$};
\node[inner sep=1,draw,shape=circle] at (1.5*\x,2.5*\x) {\small $6$};

\end{scope}

\end{tikzpicture}
\caption{The tiling with $\alpha\beta^2, \alpha^2\gamma^2$}
\label{Fig-a3-abab-Tiling-albe2-al2ga2}
\end{figure}
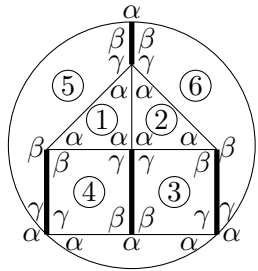

For $\AVC = \{ \alpha\beta^2, \alpha^2\gamma^2, \alpha^4, \alpha^5 \}$, the vertex $\alpha^2\gamma^2$ determines $T_1, T_2, T_3, T_4$ in the second picture of Figure \ref{Fig-a3-abab-Tiling-albe2-al2ga2}. Then $\alpha_1\beta_4\cdots=\alpha\beta^2$ determines $T_5$ and by mirror symmetry we also get $T_6$. By $\beta_3\beta_4\cdots, \beta_5\beta_6\cdots=\alpha\beta^2$, we determine the remaining two tiles. 

The tiling in the second picture of Figure \ref{Fig-a3-abab-Tiling-albe2-al2ga2} is given by a rhombus prism with triangulated rhombi. For  $x=y$, the same argument in the previous case gives
\begin{align*}
&\alpha = 4 \tan^{-1} (\sqrt{17} - 4)^{\frac{1}{2}},&
&\beta = \pi - 2 \tan^{-1} (\sqrt{17} - 4)^{\frac{1}{2}},& \\
&\gamma = \pi - 2 \tan^{-1} (\sqrt{17} - 4)^{\frac{1}{2}},& 
&x = 2 \tan^{-1} ( 1 + \tfrac{1}{2}( \sqrt{17} - 5 ) )^{\frac{1}{2}}.&
\end{align*}
The approximate values to the above are 
\begin{align*}
\alpha \approx 0.42965\pi, \quad
\beta \approx 0.78518\pi, \quad
\gamma \approx 0.57035\pi, \quad
x \approx 0.40941\pi.
\end{align*}
The case of $x\neq y$ is a perturbation of the equality case and the identities between them and the angles can be seen in \cite{ch}. \qedhere
\end{case*}

\end{proof}

\begin{prop}\label{Prop-a3-abab-al2be2}  The dihedral tiling with $\alpha^2\beta^2$ is the last tiling in Figure \ref{Fig-a3-abab-Sporadic-Tilings}.
\end{prop}

\begin{proof} By $\alpha^2\beta^2$ and $\beta+\gamma > \pi$, we have $\gamma>\alpha$. It implies $\beta>\gamma>\alpha$. Then $\alpha^2\beta^2$ implies $\beta>\frac{1}{2}\pi>\alpha$.

By $\alpha^2\beta^2$ and \eqref{Eq-a3-abab-be2}, we get $\beta^2\cdots=\alpha^2\beta^2$. By $\alpha^2\beta^2$ and $\beta>\gamma$, we know that $\alpha\beta\gamma, \alpha^2\beta\gamma$ are not vertices. By $\alpha+\beta=\pi$ and $3\gamma>\pi$, we have $\alpha+\beta+3\gamma>2\pi$. This implies that $\alpha\beta\gamma^3, \alpha^2\beta\gamma^3, \alpha^2\beta\gamma^5$ are not vertices. So \eqref{Eq-a3-abab-bega} implies that $\beta\gamma\cdots$ is not a vertex. Hence $\beta\cdots=\beta^2\cdots=\alpha^2\beta^2$.

By Parity Lemma and \eqref{Eq-a3/am-abab-ga2} and no $\beta\gamma\cdots$, we get $\gamma\cdots=\gamma^2\cdots=\alpha^{a\ge2}\gamma^c$. By $\alpha^2\beta^2$ and $\beta>\gamma$, we know that $\alpha^2\gamma^2$ is not a vertex. Then Parity Lemma and $\gamma>\alpha>\frac{1}{3}\pi$ imply $\gamma\cdots=\alpha^3\gamma^2$.

With the knowledge of $\beta\cdots, \gamma\cdots$, the only remaining vertex is $\alpha^a$. By $\frac{1}{2}\pi>\alpha>\frac{1}{3}\pi$, we have $\alpha^a=\alpha^5$. Hence we get
\begin{align*}
\AVC = \{ \alpha^2\beta^2, \alpha^5, \alpha^3\gamma^2 \}.
\end{align*}
From the above, we have $\beta\cdots=\alpha^2\beta^2$ and $\gamma\cdots=\alpha^3\gamma^2$ and $\alpha^4\cdots=\alpha^5$.

A pair of $\alpha^2\beta^2, \alpha^3\gamma^2$ determine tiles $T_1, T_2, ..., T_7$ in Figure \ref{Fig-a3-abab-Tiling-al2be2-al3ga2}. In the first picture, $\alpha_7\beta_2\cdots=\alpha^2\beta^2$ and $\alpha_3\gamma_2\cdots=\alpha^3\gamma^2$ determine $T_8, T_9, T_{10}$. Then the same pattern repeats. Then $\alpha_6\alpha_7\alpha_9\alpha_{12}\cdots=\alpha^5$ gives the centre triangle in the second picture. Similarly, $\alpha_3\alpha_4\alpha_{11}\cdots=\alpha^4\cdots=\alpha^5$ gives the triangle in the exterior of the second picture.

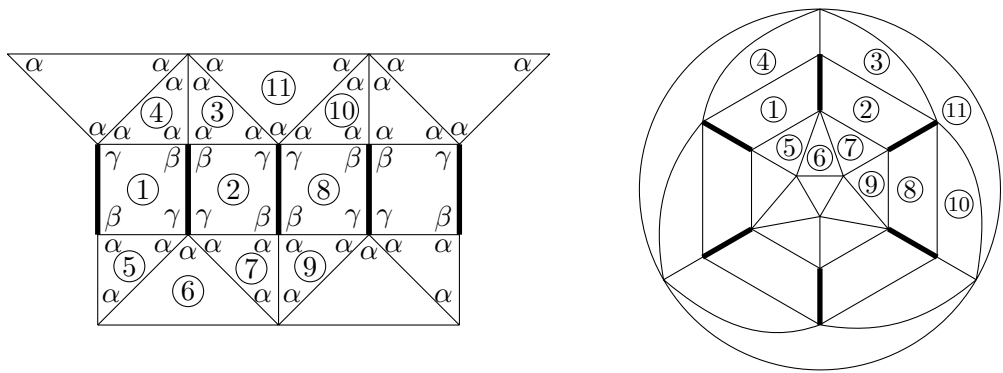
\begin{figure}[h!] 
\centering
\begin{tikzpicture}

\tikzmath{
\s=1;
}

\begin{scope}[] 

\tikzmath{
\s=1;
\r=0.85;
\th=360/4;
\x = \r*cos(0.5*\th);
\tz=4;
\tzz=\tz-1;
}

\foreach \aa in {0,...,\tzz} {

\tikzset{shift={(\aa*2*\x,0)}}

\foreach \a in {0,...,3} {

\draw[rotate=\a*\th]
	(0.5*\th:\r) -- (1.5*\th:\r)
;

}

\draw[line width=2]
	(0.5*\th:\r) -- (-0.5*\th:\r)
	(1.5*\th:\r) -- (2.5*\th:\r)
;

}

\foreach \aa in {1,3,...,\tz} {

\draw[shift={(\aa*2*\x,0)}]
	(-\x, \x) -- (-\x, 3*\x)
;

}

\foreach \aa in {0,2,...,\tz} {

\draw[shift={(\aa*2*\x,0)}]
	(-\x, -\x) -- (-\x, -3*\x)
;

}

\foreach \aa in {-1,1,...,\tz} {

\draw[shift={(\aa*2*\x,0)}]
	(\x, \x) -- (3*\x, 3*\x)
	(-\x, 3*\x) -- (\x, \x)
	(-\x, 3*\x) -- (3*\x, 3*\x)
;

\node at (-0.4*\x+\aa*2*\x,2.75*\x) {\small $\alpha$};
\node at (2.4*\x+\aa*2*\x,2.75*\x) {\small $\alpha$};

}

\foreach \aa in {0,2,...,\tzz} {

\draw[shift={(\aa*2*\x,0)}]
	(\x, -\x) -- (-\x, -3*\x)
	(\x, -\x) -- (3*\x, -3*\x)
	(-\x, -3*\x) -- (3*\x, -3*\x)
;

\node at (1*\x+\aa*2*\x,-1.4*\x) {\small $\alpha$};

}

\foreach \aa in {0,2,...,\tz} {

\node at (-\x+\aa*2*\x,1.35*\x) {\small $\alpha$};

}

\foreach \aa in {0,2,...,\tzz} {

%

%
\node at (-0.45*\x+\aa*2*\x,1.25*\x) {\small $\alpha$};
\node at (0.65*\x+\aa*2*\x,1.25*\x) {\small $\alpha$};
\node at (1.35*\x+\aa*2*\x,1.25*\x) {\small $\alpha$};
\node at (2.45*\x+\aa*2*\x,1.25*\x) {\small $\alpha$};
\node at (1.3*\x+\aa*2*\x,2.35*\x) {\small $\alpha$};
\node at (0.7*\x+\aa*2*\x,2.35*\x) {\small $\alpha$};

\node at (-0.65*\x+\aa*2*\x,-1.25*\x) {\small $\alpha$};
\node at (0.45*\x+\aa*2*\x,-1.25*\x) {\small $\alpha$};
\node at (1.55*\x+\aa*2*\x,-1.25*\x) {\small $\alpha$};
\node at (2.65*\x+\aa*2*\x,-1.25*\x) {\small $\alpha$};
\node at (2.65*\x+\aa*2*\x,-2.35*\x) {\small $\alpha$};
\node at (-0.7*\x+\aa*2*\x,-2.35*\x) {\small $\alpha$};

}

\foreach \aa in {0,2,...,\tzz} {

\foreach \a in {0,1} {

\tikzset{shift={(0+\aa*2*\x,0)}, rotate=\a*180}

\node at (-0.65*\x,0.65*\x) {\small $\gamma$};
\node at (0.65*\x,0.65*\x) {\small $\beta$};

}

\foreach \a in {0,1} {

\tikzset{shift={(2*\x+\aa*2*\x,0)}, rotate=\a*180}

\node at (-0.65*\x,0.65*\x) {\small $\beta$};
\node at (0.65*\x,0.65*\x) {\small $\gamma$};

}

}

\node[inner sep=1,draw,shape=circle] at (0,0) {\small $1$};
\node[inner sep=1,draw,shape=circle] at (2*\x,0) {\small $2$};
\node[inner sep=1,draw,shape=circle] at (1.65*\x,1.72*\x) {\small $3$};
\node[inner sep=1,draw,shape=circle] at (0.3*\x,1.7*\x) {\small $4$};
\node[inner sep=1,draw,shape=circle] at (-0.3*\x,-1.7*\x) {\small $5$};
\node[inner sep=1,draw,shape=circle] at (\x,-2.25*\x) {\small $6$};
\node[inner sep=1,draw,shape=circle] at (2.4*\x,-1.75*\x) {\small $7$};

\node[inner sep=0.25,draw,shape=circle] at (\x+2*\x,2.25*\x) {\footnotesize $11$};

\node[inner sep=1,draw,shape=circle] at (0+4*\x,0) {\small $8$};
\node[inner sep=1,draw,shape=circle] at (-0.3*\x+4*\x,-1.7*\x) {\small $9$};
\node[inner sep=0.25,draw,shape=circle] at (0.4*\x+4*\x,1.75*\x) {\footnotesize $10$};

\end{scope} 

\begin{scope}[xshift=9*\s cm]

\tikzmath{
\r=0.6;
\g=6;
\ph=360/\g;
\x=\r*cos(\ph/2);
\y=\r*sin(\ph/2);
\rr=2*\y/sqrt(2);
\h=3;
\th=360/\h;
\xx=\rr*cos(\th/2);
\R=2*\r/cos(60);
}


	(90:3*\r) -- (90+\ph:3*\r) --(90+2*\ph:3*\r) --(90+3*\ph:3*\r) -- (90+4*\ph:3*\r) -- (90+5*\ph:3*\r) -- cycle
;

	(90:1.5*\r) -- (90+\ph:1.5*\r) --(90+2*\ph:1.5*\r) --(90+3*\ph:1.5*\r) -- (90+4*\ph:1.5*\r) -- (90+5*\ph:1.5*\r) -- cycle
;

\foreach \t in {0,...,2} {

\draw[rotate=\t*\th]
	(90-\ph:0.6*\r) -- (90+\ph:0.6*\r)
	(90-\ph:0.6*\r) -- (90-\ph:1.75*\r) 
	(90-\ph:0.6*\r) -- (90:1.75*\r)
	(90+\ph:0.6*\r) -- (90:1.75*\r)
	(90:3*\r) -- (90:\R)
	(90:\R) to[out=200, in=80] (90+\ph:3*\r) 
	(90:\R) to[out=-20, in=110] (90-\ph:3*\r) 
	(90:\R) arc (90:90+\th:\R)
;

}

\foreach \p in {0,...,5} {

\draw[rotate=\p*\ph]
	(90-1*\ph:1.75*\r) -- (90:1.75*\r)
	(90:1.75*\r) -- (90:3*\r)
	(90-1*\ph:3*\r) -- (90:3*\r)
;

\draw[line width=2, rotate=\p*\ph]
	(90:1.75*\r) -- (90:3*\r)
;

}

\node[inner sep=0.5,draw,shape=circle] at (90+0.5*\ph:2.05*\r) {\footnotesize $1$};
\node[inner sep=0.5,draw,shape=circle] at (90-0.5*\ph:2.05*\r) {\footnotesize  $2$};
\node[inner sep=0.5,draw,shape=circle] at (90-0.4*\ph:3.1*\r) {\footnotesize  $3$};
\node[inner sep=0.5,draw,shape=circle] at (90+0.4*\ph:3.1*\r) {\footnotesize  $4$};
\node[inner sep=0.5,draw,shape=circle] at (90+0.6*\ph:1.15*\r) {\footnotesize  $5$};
\node[inner sep=0.5,draw,shape=circle] at (90:0.7*\r) {\footnotesize  $6$};
\node[inner sep=0.5,draw,shape=circle] at (90-0.6*\ph:1.15*\r) {\footnotesize  $7$};
\node[inner sep=0.5,draw,shape=circle] at (90-1.5*\ph:2*\r) {\footnotesize  $8$};
\node[inner sep=0.5,draw,shape=circle] at (90-1.4*\ph:1.15*\r) {\footnotesize  $9$};
\node[inner sep=0.25,draw,shape=circle] at (90-1.6*\ph:3.1*\r) {\scriptsize $10$};
\node[inner sep=0.25,draw,shape=circle] at (90-\ph:3.5*\r) {\scriptsize $11$};

\end{scope}

\end{tikzpicture}
\caption{The tiling with $\alpha^2\beta^2, \alpha^5, \alpha^3\gamma^2$}
\label{Fig-a3-abab-Tiling-al2be2-al3ga2}
\end{figure}

Hence we get the last tiling in Figure \ref{Fig-a3-abab-Sporadic-Tilings}. It is given by the hexagonal prism with triangulated hexagons.

The angle sums of $\alpha^2\beta^2, \alpha^5, \alpha^3\gamma^2$ and $\cos x = \cos^2 x + \sin^2 x \cos \alpha$ imply
\[
\alpha = \tfrac{2}{5}\pi, \quad
\beta = \tfrac{3}{5}\pi, \quad
\gamma = \tfrac{2}{5}\pi, \quad
x = \cos^{-1} \tfrac{1}{\sqrt{5}}.
\qedhere
\]
\end{proof}

\end{document}